\documentclass[leqno]{article}
\usepackage{amsfonts,amsmath,amssymb,amsthm}

\usepackage{setspace}
\usepackage{tabularx}
\usepackage{hyperref}

\usepackage{fancyhdr}
\usepackage[english]{babel}
\usepackage{footmisc}
\usepackage{algorithmic}
\usepackage{listings}
\usepackage{fancyvrb}
\usepackage{booktabs}
\usepackage{epstopdf}

\usepackage{mathptmx}
\usepackage{amsmath,amsthm}
\usepackage{amssymb,amsfonts}
\usepackage{tikz}
\usepackage{pgf}
\usepackage{scalefnt}
\usepackage{pgfplots}

\usepackage{multicol}
\usepackage{placeins}
\usepackage{caption}[2004/07/16]
\usepackage{subcaption} 
\usepackage{graphicx}
\usepackage{mathdots}
\usepackage{cleveref}
\usepackage{mathrsfs}
\usepackage{txfonts}
\usepackage{amsopn}

\usepackage{stmaryrd}
\usepackage{array}
\usepackage{longtable}
\usepackage{tabu}
\usepackage{txfonts}
\usepackage{geometry}
\usepackage{supertabular}
\usepackage{lscape} 
\makeatletter

\newcommand{\Rmnum}[1]{\expandafter\@slowromancap\romannumeral #1@}
\makeatother
\pgfmathdeclarefunction{gauss}{2}{%
    \pgfmathparse{1/(#2*sqrt(2*pi))*exp(-((x-#1)^2)/(2*#2^2))}%
        }
\usepackage{numprint}

\nprounddigits{15}

\graphicspath{{figs/}}  % Location of the graphics files

\usetikzlibrary{arrows,automata, positioning, calc, matrix, decorations.pathreplacing}
\hypersetup{%
	pdfstartview={Fit}
}

\usepackage{graphicx,epsfig}
\usepackage{checkend}
\usepackage{longtable,geometry}
\newtheorem{definition}{Definition}[section]
\newtheorem{theorem}{Theorem}
\newtheorem{lemma}[theorem]{Lemma}
\newtheorem{remark}{Remark}
\numberwithin{equation}{section}

\begin{document}
\begin{center}  {\huge\textbf{Evolution of interfaces for the non-linear parabolic p-Laplacian type reaction-diffusion equations. II. Fast diffusion vs. absorption }}
%\par\medskip\bigskip \huge\textsc{Project Description}
\par \medskip\bigskip\end{center}
\begin{center} {\Large\textsc{Ugur G. Abdulla and Roqia Jeli}}
\par \medskip\bigskip\end{center}
\begin{center} {\large\noindent \textsc{Department of Mathematics, Florida Institute of Technology, Melbourne, Florida 32901}}
\par \medskip\bigskip\end{center}

%\title{Evolution of interfaces for the non-linear parabolic p-Laplacian type diffusion equation of non-Newtonian elastic filtration with strong absorption}

%\shorttitle{interfaces for the non-Newtonian elastic filtration with absorption} %%%for recto running head
%\shortauthorlist{U.G. Abdulla \& R. Jeli} %%% for verso running head

%\author{%%%% First author details
%\name{Ugur G. Abdulla$^*$}
%\address{Department of Mathematical Sciences, Florida Institute of Technology, Melbourne, FL 32901\email{$^*$Corresponding author: abdulla@fit.edu}}
%%%%%%% Second author details
%\name{Roqia A. Jeli}
%\address{Department of Mathematical Sciences, Florida Institute of Technology, Melbourne, FL 32901}
%%%%%%%
%}

%\maketitle

\begin{abstract}
{
We present a full classification of the short-time behaviour of the interfaces and local solutions to the nonlinear parabolic $p$-Laplacian type reaction-diffusion equation of non-Newtonian elastic filtration
\[ u_t-\Big(|u_x|^{p-2}u_x\Big)_x+bu^{\beta}=0, \ 1<p<2, \beta >0 \]
If the interface is finite, it may expand, shrink, or remain stationary as a result of the competition of the diffusion and reaction terms near the interface, expressed in terms
 of the parameters $p,\beta, sign~b$, and asymptotics of the initial function near its support. In some range of parameters, strong domination of the diffusion causes infinite speed of propagation and interfaces are absent.
 In all cases with finite interfaces we prove the explicit formula for the interface and the local solution with accuracy up to constant coefficients. We prove explicit asymptotics of the local solution at infinity in all cases with infinite speed of propagation.
 The methods of the proof are based on nonlinear scaling laws, and a barrier technique using special comparison theorems in irregular domains with characteristic boundary curves. A full description of small-time behaviour of the interfaces and local solutions
 near the interfaces for slow diffusion case when $p>2$ is presented in a recent paper {\it Abdulla \& Jeli, Europ. J. Appl. Math. 28, 5(2017), 827-853.}}
 \end{abstract}

%{\bf Key words:} { $p$-Laplacian type reaction-diffusion equations; Evolution of interfaces; Fast diffusion; Nonlinear degenerate parabolic equations.}
%%%% If classification number provided then

%{\bf AMS subject classifications:} 34K30, 35K57, 35Q80,  92D25

\section{Introduction}\label{intro.}
Consider the Cauchy problem (CP) for the $p$-Laplacian type reaction-diffusion equation
\begin{equation}\label{C1}
Lu\equiv u_t-\Big(|u_x|^{p-2}u_x\Big)_x+bu^{\beta}=0, \ x\in \mathbb{R}, 0<t<T,
\end{equation}
with
\begin{equation}\label{C2}
u(x,0)=u_0(x),~~x\in \mathbb{R},
\end{equation}
where~~$1<p<2,\;b\in \mathbb{R},\;\beta>0,\;0<T\leq +\infty$ and $u_0$ is non-negative and continuous. Throughout the paper we assume that either $b\geq 0$ or $b<0$ and $\beta\geq 1$ (see Remark~\ref{nonunique}). \eqref{C1} is called an equation of non-Newtonian elastic filtration with absorption or reaction \cite{Barenblatt1, kalashnikov1987some}. The goal of the paper is to present a full classification of the short-time behavior of the interfaces and local solutions near the interfaces and at infinity in a CP with a compactly supported initial function. The key ingredient of the equation \eqref{C1} is to model competition between the fast diffusion force with infinite speed of propagation property (\cite{Barenblatt1, Barenblatt2}) and absorption or reaction term. Without loss of generality, it is assumed that $\eta(0)=0,$ where
\[\eta(t)=\text{sup}~\{x:u(x,t)>0\}.\]
More precisely, in all cases with finite interfaces we are interested in the short-time behavior of the interface function ~$\eta(t)$~and of the local solution near the interface. In all cases with infinite speed of propagation, we aim to classify the asymptotics of the solution at infinity. We use the notation
\[ f(y)\sim g(y),~~\text{as} \quad y\rightarrow y_o \]
instead of
\[ \lim\limits_{y\rightarrow y_o} \frac{f(y)}{g(y)}=1. \]
Furthermore, unless otherwise stated, we shall assume that 
\begin{equation}\label{C3}
u_0\sim C(-x)_+^{\alpha}~~\text{as}~x\rightarrow 0^-~~~\text{for some}~~~C>0,~\alpha>0,\end{equation}
where $(\cdot)_+=\max(\cdot;0)$. The behaviour of $u_0$~as~$x\rightarrow -\infty$~has no influence on our results.  Accordingly, we may suppose that ~$u_0$~either is bounded or unbounded with growth condition as ~$x\rightarrow -\infty,$ which is suitable for existence, uniqueness, and comparison results. In some cases we will consider the special case
\begin{equation}\label{C4}
u_0(x)=C(-x)_+^{\alpha},\quad x\in \mathbb{R}.\end{equation}
Precisely, that will be done in all cases when the solution to the CP \eqref{C1}, \eqref{C4} is of self-similar form. In these cases our estimations will be global in time. 

A full classification of the small-time behavior of $\eta(t)$ and of the local solution near ~$\eta(t)$ depending on the parameters ~$p,b,\beta, C,$~and~$\alpha$ in the case of slow diffusion ($p>2$) is presented in a recent paper \cite{Abdulla35}. A similar classification for the reaction-diffusion equation
\begin{equation}\label{nonldiff}
u_t-(u^m)_{xx}+bu^\beta=0
\end{equation}
is presented in \cite{Abdulla1} for the slow diffusion case ($m>1$), and in \cite{Abdulla2} for the fast diffusion case ($0<m<1$). %in this paper a similar description is presented for the fast diffusion ($1<p<2$) case; we also give a full description of local solution in cases in which interfaces are absent. %The behavior of $u_0$~as~$x\rightarrow -\infty$~has no influence on our results.  Accordingly, we may suppose that ~$u_0$~either is bounded or unbounded without any conditions  as ~$x\rightarrow -\infty,$ which is suitable for existence, uniqueness, and comparison results (see Section \ref{sec: preliminary results}). 
The semilinear case ($p=2$ in \eqref{C1}) was analyzed in \cite{Grundy1,Grundy2}. It should be noted that as in the case of PDE \eqref{nonldiff}, the semilinear case is a singular limit of the general case. For instance, if $0<\beta<1,\; p-1>\beta,\; C>0,\; \alpha<\frac{p}{p-1-\beta}$, then the interface initially expands and if $p>2$ then \cite{Abdulla35}
\[\eta(t)\sim C_1t^{1/(p-\alpha(p-2))}~~\text{as}~t\rightarrow 0^+,\]
while if $p<2,$ we prove below that
\[\eta(t)\sim C_2t^{(p-1-\beta)/p(1-\beta)}~~\text{as}~t\rightarrow 0^+.\]
Formally, as ~$p\rightarrow 2$ both estimates yield a false result, and from \cite{Grundy2} it follows that if $p=2$,~then 
\[\eta(t)\sim C_3(t~\text{ log}~ 1/t)^{\frac{1}{2}}\]
($C_i, i=\overline{1,3}$~ are positive constants). %The method of this paper are similar to those of \cite{Abdulla35, Abdulla1, Abdulla2}. %As in \cite{Abdulla35}, the results of the recent paper \cite{Abdulla3} on the general theory of initial-boundary value problems for reaction-diffusion equations in non-cylindrical domains with non-smooth and characteristic boundary curves will play a crucial role. 

The mathematical theory of nonlinear p-Laplacian type degenerate parabolic equations is well developed (see \cite{dibe-sv}).
Throughout this paper we shall follow the definition of weak solutions and of supersolutions (or subsolutions) of the CP \eqref{C1},\eqref{C2} in the following sense:
\begin{definition}
A measurable function $u\geq 0$ is a weak solution (respectively sub- or supersolution) of the CP \eqref{C1}. \eqref{C2} in $\mathbb{R} \times [0,T]$ if 
\begin{itemize}
\item  $u \in C_{loc}(0,T; L^2_{loc}(\mathbb{R})) \cap L^p_{loc}(0,T; W_{loc}^{1,p}(\mathbb{R})\cap L^{1+\beta}_{loc}(\mathbb{R}))$
\item For $\forall$ subinterval $[t_0, t_1] \subset (0,T]$ and for $\forall \mu_i \in C^1[t_0, t_1], \ i=1,2$ such that $\mu_1(t) < \mu_2(t)$ for $t\in [t_0, t_1]$
\begin{equation}\label{supersub}
\int_{\mu_1(t)}^{\mu_2(t)}u \phi dx \Big |_{t_0}^{t_1}+\int_{t_0}^{t_1}\int_{\mu_1(t)}^{\mu_2(t)}(-u\phi_t+|u_x|^{p-2}u_x\phi_x+bu^\beta \phi )dx dt = 0 \  ( resp.  \leq  or \geq \  0)
\end{equation}
where $\phi \in C_{x,t}^{2,1}(\overline{D})$ is an arbitrary function (respectively nonnegative function) that equals zero when $x=\mu_i(t), t_0\leq t \leq t_1, i=1,2$, and
\[ D=\{(x,t): \mu_1(t)<x<\mu_2(t), t_0< t <t_1\}  \]
\item $\lim\limits_{t\downarrow 0} u(x,t)=u_0(x),~~\text{for all}~~x\in\mathbb{R}$
\end{itemize}
\end{definition}
The questions of existence and uniqueness of initial boundary value problems for \eqref{C1}, comparison theorems, and regularity of weak solutions are known due to \cite{dibe1,  dibe2, dibe3, dibe-sv,est-vazquez,kalashnikov2,kalashnikov3, tsutsumi} etc. Qualitative properties of free boundaries for the quasilinear degenerate parabolic equations were studied via energy methods in \cite{diaz,ADS}. It is proved in \cite{dibe3} that existence, uniqueness, and comparison theorems are valid for the CP \eqref{C1},\eqref{C2} with $b=0$, $1<p<2$ without any growth condition on the initial function $u_0$ at infinity. In particular, $\alpha >0$ is arbitrary in \eqref{C4}. The same results are true of the CP  \eqref{C1},\eqref{C2}  with $b>0$ (\cite{dibe-sv}). This follows from the fact that the solution of the CP \eqref{C1}, \eqref{C2} with $b=0$ is a supersolution of the CP with $b>0$, and hence it becomes a global locally bounded uniform upper bound for the increasing sequence of approximating bounded solutions of the CP with $b>0$.

The organization of the paper is as follows: In Section \ref{sec: description of the main results}, we outline the main results. Section \ref{sec: details of the main results} describes some further technical details of the main results. In Section \ref{sec: preliminary results}, we then apply scale of variables methods for some preliminary estimations which are necessary for using our barrier technique. Finally in Section \ref{sec: proofs of the main results.} we prove the results of Section \ref{sec: description of the main results}. 
To avoid difficulties for the reader we give explicit values of some of constants which appear in Sections \ref{sec: description of the main results}, \ref{sec: details of the main results}  and \ref{sec: proofs of the main results.} in the appendix.
\begin{remark}\label{nonunique} We are not considering the case $b<0, 0<\beta<1$ in this paper due to the fact that in general, uniqueness and comparison theorems don't hold for the solutions of the Cauchy problem \eqref{C1},\eqref{C2}. Although the methods of this paper can be applied to identify asymptotic properties of the minimal solution at infinity in this case. The methods of this paper can be applied to similar problem for the non-homogeneous reaction-diffusion equations with space and time variable dependent power type coefficients (\cite{shmarev2015interfaces}). It should be also mentioned that modification of the method can be applied to radially symmetric solutions of the multidimensional $p$-Laplacian type reaction-diffusion equation
\[ u_t=div(|\nabla u|^{p-2}\nabla u)+bu^\beta. \]
\end{remark}

\section{Main Results}\label{sec: description of the main results}
Throughout this section we assume that $u$ is a unique weak solution of the CP \eqref{C1}-\eqref{C3}. There are five different subcases, as shown in Fig. \ref{fig:1}. The main results are outlined below in Theorems 1, 2, 3, 4 and 5 corresponding directly to the cases $I,II,III.IV$ and $V$ in Fig. \ref{fig:1}. 

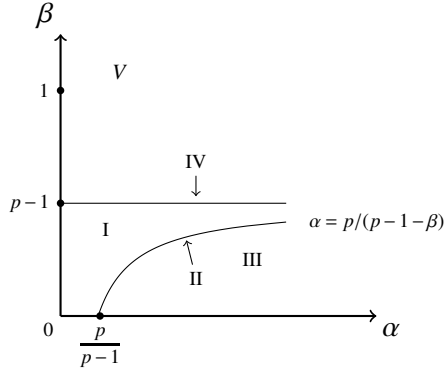
\begin{figure}
		\begin{center}	
		{\scalefont{0.75}
\begin{tikzpicture}[xscale=1.5, yscale=0.5]
\draw[->, thick] (0,0)--(2.8,0);
\draw[->,thick] (0,0)--(0,7.5);
\node[below left] at (0,0) {$0$};
\node[fill, shape=circle, label=180:{$p-1$}, inner sep=1pt] at (0,3) {};
\node[fill, shape=circle, label=180:{$1$}, inner sep=1pt] at (0,6) {};
\node[fill, shape=circle, label=-90:{$\displaystyle  \frac{p}{p-1}$}, inner sep=1pt] at (.35,0) {};
\node[left] at (0,8) {\scalefont{1.5} $\beta$};
\node[below] at (2.9,0) {\scalefont{1.5} $\alpha$};
\draw[domain=0:2.5, smooth] plot ({1/(3-\x)},\x);
\node at (2.8,2.5) {$\alpha =p/(p-1-\beta)$};
\node at (0.5,6.5) {\scalefont{1.1} $V$};
\node (1) at (0.4,2.3) {I};
\node (2) at (1.2,1) {II};
\node (3) at (1.7,1.5) {III};
\node (4) at (1.2,4.1) {IV};
\node (*) at (1.2,3) {};
\node (**) at (.71,1.5) {};
\node (***) at (1.1,2.2) {};
\draw[ ->][ left]
  (4) edge (*);
   \draw[ ->][ left]
  (2) edge (***);
\draw (0,3)--(2,3);
\draw ( (0.5,1) {};
\end{tikzpicture}
}
		\caption{Classification of different cases in the ($\alpha$,$\beta$) plane for interface development in problem \eqref{C1}-\eqref{C4}.}
		\label{fig:1}
		\end{center}
		\end{figure}

%\begin{figure}
%    \centering
%    \includegraphics[width = 0.7\textwidth]{Classification1.png}
%    \caption{Classification of different cases in the ($\alpha$,$\beta$) plane for interface development in problem \eqref{C1}-\eqref{C4}.}\label{fig:1}
%\end{figure}
\begin{theorem}\label{diffusiondominates'}
Let $b>0, 0<\beta<p-1,\; 0<\alpha<p/(p-1-\beta)$. Then, the interface initially expands and for some positive $\delta>0$ 
%\begin{subequations}\label{C3'ab}
\begin{equation}\label{C3'b}
\zeta_1t^{(p-1-\beta)/p(1-\beta)}\leq \eta(t)\leq \zeta_2t^{(p-1-\beta)/p(1-\beta)},\; 0<t\leq \delta,
\end{equation} %\end{subequations}
(see Appendix for explicit values of $\zeta_1,\;\zeta_2$). Moreover, for arbitrary $\rho\in \mathbb{R}$, there exists a positive number $f(\rho)$ depending on $C,p$ and $\alpha$ such that 
\begin{equation}\label{C1'}
u(\xi_{\rho}(t),t) \sim f(\rho) t^{\frac{\alpha}{p+\alpha (2-p)}}~ \text{as} \quad t\rightarrow 0^+
\end{equation}
where $\xi_{\rho}(t)=\rho t^{1/(p+\alpha(2-p))}$.
\end{theorem}
\begin{theorem}\label{diffusionbalance'}
Let ~$b>0, 0<\beta <p-1,\; \alpha=p/(p-1-\beta)$ and 
\begin{equation}\label{C*}
C_*=\Big[(b\left|p-1-\beta\right|^p)/((1+\beta)(p-1)p^{p-1})\Big]^{1/(p-1-\beta)}.
\end{equation}
Then the interface expands or shrinks accordingly as $C>C_*$ or $C<C_*$ and
\begin{equation}\label{C8'}
\eta(t) \sim\zeta_*t^{(p-1-\beta)/p(1-\beta)},\qquad \text{as} \ t\rightarrow 0^+,\end{equation} where $\zeta_* \lessgtr 0$ if $C\lessgtr C_*$, and for arbitrary $\rho<\zeta_*$ there exists $f_1(\rho)>0$ satisfies \begin{equation}\label{C9'}
u(\zeta_{\rho}(t),t)\sim f_1(\rho)t^{1/(1-\beta)}\qquad \text{as} \quad t\rightarrow 0^+,
\end{equation}
where $\zeta_{\rho}(t)=\rho t^{\frac{p-1-\beta}{p(1-\beta)}}$.
 \end{theorem}
\begin{theorem}\label{reactiondominates'}
Let $b>0,\;0<\beta<p-1,\;\alpha >p/(p-1-\beta)$. Then interface shrinks and
\begin{equation}\label{C10'}
\eta(t)\sim -\ell_* t^{1/\alpha(1-\beta)}\qquad\text{as}~t\rightarrow 0^+,
\end{equation}
where~$\ell_*=C^{-1/\alpha}(b(1-\beta))^{1/\alpha(1-\beta)}.$~For arbitrary $\ell>\ell_*$, we have \begin{equation}\label{C11'}
u(\eta_l(t),t)\sim [C^{1-\beta}\ell^{\alpha(1-\beta)}-b(1-\beta)]^{1/(1-\beta)}t^{1/(1-\beta)}\qquad\text{as}~t\rightarrow 0^+,
\end{equation}
where $\eta_l(t)=-lt^{1/\alpha(1-\beta)}.$
\end{theorem} 
\begin{theorem}\label{infinitespeed}
 Let $b>0,\; 0<\beta=p-1<1,\; \alpha>0$. Then there is an infinite speed of propagation %and  \eqref{C1'}, is valid (see Lemma \ref{lemma3.2}). 
and $\forall~\epsilon>0,\; \exists~ \delta=\delta(\epsilon)>0$ such that 
\begin{equation}\label{C12'}
t^{1/(2-p)}\phi(x)\leq u(x,t) \leq  (t+\epsilon)^{1/(2-p)}\phi(x)\qquad \text{for}\;\;0<x<\infty,\;0\leq t\leq \delta_{\epsilon},\end{equation}
where $\phi(x)$ solves the ODE problem
\begin{subequations}\label{C13'ab}
\begin{equation}\label{C13'a}
(|\phi'(x)|^{p-2}\phi'(x))'=\frac{1}{2-p}\phi(x)+b\phi^{p-1}(x)\end{equation}
\begin{equation}\label{C13'b}\phi(0)=1,\; \phi(\infty)=0.
\end{equation}\end{subequations}
Solution $u$ satisfies the asymptotic formula  
\begin{equation}\label{logasympu'.} \log u(x,t)\sim-\Big(\frac{b}{p-1}\Big)^{1/p}x\;\; \text{as}\;x\rightarrow+\infty.\end{equation}
\end{theorem} 
\begin{theorem}\label{infinitespeed'}
Let either $b>0, \beta >p-1$ or $b<0, \beta \geq 1$ or $b=0$ and
\begin{equation}\label{d-constant}
D=\Big(2(p-1)p^{p-1}(2-p)^{1-p}\Big)^{1/(2-p)}.
\end{equation}
 Then there is an infinite speed of propagation and \eqref{C1'} is valid.
If either $b>0,\beta\geq 2/p$ or $b<0,\beta\geq 1$ or $b=0$ then $\exists \delta>0$ such that for $\forall$ fixed $t\in (0,\delta]$
\begin{equation}\label{C17'}
u(x,t)\sim Dt^{1/(2-p)}x^{p/(p-2)} \qquad \text{as}~~x\rightarrow +\infty.
\end{equation}
If $b>0, 1\leq  \beta<2/p$, then 
\begin{equation}\label{C18'}
 \underset{t\rightarrow 0+}\lim  \underset{x\rightarrow +\infty}\lim ut^{1/(p-2)}x^{\frac{p}{2-p}}=D.
\end{equation}
If $b>0,\; p-1<\beta<1$ then $\exists \delta >0$ such that for arbitrary fixed $t\in(0,\delta]$
\begin{equation}\label{C21'}
u(x,t)\sim C_*x^{p/(p-1-\beta)}\qquad \text{as}~~x\rightarrow +\infty.
\end{equation}
\end{theorem} 
\section{Further Details of the Main Results}\label{sec: details of the main results} 
%%%%%%
In this section we outline some essential details of the main results described in Theorems 1-5.  %of Section \ref{sec: description of the main results}.

{\it Further details of Theorem \ref{diffusiondominates'}}. 
%%%%%%%%%%%%%%
Solution $u$ satisfies the estimation
 \begin{equation}\label{C3'a}
C_1t^{1/(1-\beta)}(\zeta_1-\zeta)_+^{p/(p-1-\beta)}\leq u \leq C_*t^{1/(1-\beta)}(\zeta_2-\zeta)_+^{p/(p-1-\beta)},\; 0<t\leq \delta,
\end{equation}
where $\zeta=xt^{-(p-1-\beta)/p(1-\beta)}$ and the left-hand side of \eqref{C3'a} is valid for $0\leq x<+\infty,$ while the right-hand side is valid for $x\geq \ell_0t^{(p-1-\beta)/p(1-\beta)}$ and the constants $C_*,\;C_1,\;\zeta_1,\;\zeta_2$ and $\ell_0$ are positive and depend only on $p,\beta$ and $b$ (see Appendix).

A function $f$~is a shape function of the self-similar solution of \eqref{C1},\eqref{C4} with $b=0$~(see Lemma~\ref{lemma3.1}) and  \begin{equation}\label{C2'}
f(\rho)=C^{\frac{p}{p+\alpha(2-p)}}f_0\big (C^{\frac{2-p}{p+\alpha (2-p)}}\rho \big),\qquad f_0(\rho)=\omega(\rho,1),
\end{equation}
where ~$w$~is a solution of \eqref{C1}, \eqref{C4} with ~$b=0,\;C=1$. Lower and upper estimations for ~$f$~are given in \eqref{C22'}, \eqref{C23'}. If $u_0$ is defined as in \eqref{C4}, then the right-hand sides of \eqref{C3'a}, \eqref{C3'b} are valid for $0<t<+\infty.$ The explicit formula \eqref{C1'} means that the local behavior of the solution along the curves $x=\xi_{\rho}(t)$ approaching the origin coincides with that of the problem \eqref{C1}, \eqref{C4} with $b=0$. In other words, diffusion completely dominates in this region. However, domination of diffusion over the reaction fails along the curves $x=\zeta_{\rho}(t)=\rho t^{(p-1-\beta)/p(1-\beta)},\;\rho>0$  approaching the origin and the balance between diffusion and reaction in this region governs the interface, as expressed in estimations \eqref{C3'a}, \eqref{C3'b}. We  stress the fact that the constants $C_1, \zeta_1, \zeta_2$ and $\ell_0$ in \eqref{C3'a}, \eqref{C3'b} do not depend on $C$ and $\alpha$.

{\it Further details of Theorem \ref{diffusionbalance'}}.
%%%%%%%%%%%%%% 
 Assume that $u_0$ is defined by \eqref{C4}. If $C=C_*$ then $u_0$ is a stationary solution to \eqref{C1},\eqref{C4}. If $C\neq C_*$ the solution to \eqref{C1},\eqref{C4}
is of self-similar form \begin{equation}\label{C4'}
u(x,t)=t^{1/(1-\beta)}f_1(\zeta),\qquad~~ \zeta=xt^{-\frac{p-1-\beta}{p(1-\beta)}}, \;~u(\zeta,1)\end{equation}
\begin{equation}\label{C5'}
\eta(t)=\zeta_*t^{(p-1-\beta)/p(1-\beta)},\;\qquad 0\leq t<+\infty.
\end{equation} 
If $C>C_*$ then the interface expands, $f_1(0)=A_1>0$ (see Lemma \ref{lemma3.3}) and 
 \begin{subequations}\label{C6'ab}
\begin{equation}\label{C6'a}
C'(\zeta't^{(p-1-\beta)/p(1-\beta)}-x)_+^{p/(p-1-\beta)}\leq u(x,t) \leq C''(\zeta''t^{(p-1-\beta)/p(1-\beta)}-x)_+^{p/(p-1-\beta)},%\qquad 0<t<+\infty,
\end{equation}
\begin{equation}\label{C6'b}
\zeta'\leq \zeta_* \leq \zeta'',
\end{equation}\end{subequations}
where $0\leq x<+\infty, 0<t<+\infty$ and $C'=C_2,\; C''=C_*,\; \zeta'=\zeta_3,\;\zeta''=\zeta_4$ (see Appendix).

If $0<C<C_*$ then the interface shrinks. There exists a constant $\ell_1>0$ such that for arbitrary $\ell \leq -\ell_1,$ there exists a $\lambda>0$ such that

\begin{equation}\label{C7'}
u(\ell t^{\frac{p-1-\beta}{p(1-\beta)}},t)=\lambda t^{1/(1-\beta)}, \; t\geq 0.
\end{equation}
Moreover, $u$ and $\zeta_*$ satisfy \eqref{C6'ab} with  $C'=C_*,\; C''=C_3,\; \zeta'=-\zeta_5=-\ell_1+(\lambda /C_*)^{(p-1-\beta)/p}<0,\zeta''=-\zeta_6$ and the left-hand side of  \eqref{C6'a} is valid for $x\geq -\ell_1 t^{(p-1-\beta)/p(1-\beta)}$, while the right-hand side is valid for $x\geq -\ell_2 t^{(p-1-\beta)/p(1-\beta)}$ (see Appendix, Lemma \ref{lemma3.3} and \eqref{C7''}).

 In general the precise value $\zeta_*$ can be found only by solving the similarity ODE $\mathcal{L}^0f_1=0$ (see \eqref{C3'''b} below) and by calculating $\zeta_*=\sup \{\zeta:f_1(\zeta)>0\}.$

The right-hand side of \eqref{C9'} (respectively \eqref{C8'}) relates to the self-similar solution \eqref{C4'}, for which we have lower and upper bounds via \eqref{C6'ab}. If $u_0$ satisfies \eqref{C3} with 
$\alpha=p/(p-1-\beta), C= C_*$ then the small-time behavior of the interface and the local solution depends on the terms smaller than $C_*(-x)^{p/(p-1-\beta)}$ in the expansion of $u_0$ as $x\rightarrow 0-$.

It should be noted that if $ C>C_*$, then the estimation \eqref{C6'ab} coincides with the estimation (2.18) from \cite{Abdulla35}, proved for the case $\beta(p-1)<1,\; p>2$. If $ 0<C<C_*$ then the right-hand side of estimation \eqref{C6'ab} coincides with (2.18) from \cite{Abdulla35} proved for the case $\beta(p-1)<1,\; p>2$, while the left-hand side of \eqref{C6'ab} is new. It should also be noted that the left-hand side of the estimation (2.18) from \cite{Abdulla35}, proved there for the case $\beta(p-1)<1,\; p>2$, is still valid if $p\geq 2-\beta$. 

{\it Further details of Theorem \ref{reactiondominates'}}.
%%%%%%%%%%%%%% 
The interface initially coincides with that of the solution
\[\bar u(x,t)=\big [C^{1-\beta}(-x)_+^{\alpha(1-\beta)}-b(1-\beta)t\big]_+^{1/(1-\beta)}\]
to the problem 
\[\bar u_t+b\bar u^{\beta}=0,~~~~~~~~~~~\bar u(x,0)=C(-x)_+^{\alpha}.\]

{\it Further details of Theorem \ref{infinitespeed}}.
%%%%%%%%%%%%%% 
The solution of \eqref{C13'ab} is
\begin{equation}\label{integral1}
\phi(x)=F^{-1}(x),\quad 0\leq x<+\infty,
\end{equation}
where $F^{-1}(\cdot)$ is an inverse function of
\begin{equation}\label{integral2}
F(z)=\int_{z}^{1}\frac{dy}{y\big[\frac{b}{p-1}+\frac{p}{2(p-1)(2-p)}y^{2-p}\big]^{1/p}}, \ 0<z\leq 1.
\end{equation}
$\phi$ satisfies
\begin{equation}\label{logasymp'.} \log\phi(x)\sim-\Big(\frac{b}{p-1}\Big)^{1/p}x\;\; \text{as}\;x\rightarrow+\infty.\end{equation}
and the global estimation
\begin{equation}\label{phiglobal}
0<\phi(x) \leq e^{-\big(\frac{b}{p-1}\big)^{1/p}x},\;\;0\leq x<+\infty.
\end{equation}
Therefore, for any $\gamma> \big(\frac{b}{p-1}\big)^{1/p}$ we have
\begin{equation}\label{phiglobal1}
\lim_{x\rightarrow +\infty}\frac{\phi(x)}{e^{-\gamma x}}=+\infty.
\end{equation}
Respectively, the solution $u$ satisfies
\begin{equation}\label{phiglobal2}
\lim_{t\rightarrow 0^+}\lim_{x\rightarrow +\infty}u(x,t)e^{\big(\frac{b}{p-1}\big)^{1/p}x}=0,
\end{equation}
and for any $\gamma> \big(\frac{b}{p-1}\big)^{1/p}$
\begin{equation}\label{phiglobal3}
\lim_{x\rightarrow +\infty}\frac{u(x,t)}{e^{-\gamma x}}=+\infty, \;\;0<t \leq \delta.
\end{equation}

{\it Further details of Theorem \ref{infinitespeed'}}.
%%%%%%%%%%%%%% 
Let $\beta \geq 1$. Then for an arbitrary sufficiently small $\epsilon>0$ there exists a $\delta=\delta(\epsilon)>0$ such that
\begin{equation}\label{C15'}
 C_5t^{\alpha/(p+\alpha(2-p))}(\xi_1+\xi)^{\frac{p}{p-2}}\leq u\leq C_6t^{\alpha/(p+\alpha(2-p))}(\xi_2+\xi)^{\frac{p}{p-2}}\qquad x\geq 0,~~0\leq t\leq \delta,
\end{equation}
where $\xi=xt^{-1/(p+\alpha(2-p))}$ (see Appendix for the relevant constants). If $b>0,\; \beta\geq 1$, then the following upper estimation is also valid 
\begin{equation}\label{C16'}
u(x,t)\leq Dt^{1/(2-p)}x^{p/(p-2)}\qquad 0<x<+\infty,~~0<t<+\infty,
\end{equation}

Let $b<0,\; \beta \geq 1$. Then for an arbitrary sufficiently small $\epsilon>0$ there exists $\delta =\delta(\epsilon)>0$ such that
\begin{equation}\label{C19'}
u(x,t)\leq D(1-\epsilon)^{1/(2-p)}t^{1/(2-p)}x^{p/(p-2)}\qquad \text{for}~~\mu t^{1/(p+\alpha(2-p))}<x<+\infty,~~0<t\leq \delta,
\end{equation}
with
\[\mu =\big(D^{-1}(A_0+\epsilon)\big)^{(p-2)/p}(1-\epsilon)^{-1/p}.\]
From \eqref{C15'} and \eqref{C19'}, \eqref{C17'} again follows.

Let $b>0,\; p-1<\beta<1$. Then there exists a number $\delta>0$ such that 
\begin{equation}\label{C20'}
C_*(1-\epsilon)t^{1(1-\beta)}(\zeta_8+\zeta)_+^{p/(p-1-\beta)}\leq u(x,t)\leq C_*x^{p/(p-1-\beta)} \qquad ~0<x<+\infty,\; 0<t\leq\delta. 
\end{equation}
where $\epsilon>0$ is an arbitrary sufficiently small number

As in the case I, the explicit formula \eqref{C1'} expresses the domination of diffusion over the reaction. If $\beta \geq 1$, then from \eqref{C15'}, \eqref{C17'}, \eqref{C18'} it follows that domination of diffusion is the case for $x\gg 1$ as well, and the asymptotic behavior as $x\rightarrow +\infty$ coincides with that of the solution to problem \eqref{C1}, \eqref{C4} with $b=0$ (see below). However, if $p-1<\beta <1$ then domination of the diffusion fails for $x\gg 1$ and there is a solution of \eqref{C1} on the right-hand side of \eqref{C21'}.

Let $b=0$. In this case there is an infinite speed of propagation. First, assume that $u_0$ is defined by \eqref{C4}. Then the solution to \eqref{C1}, \eqref{C4} has the self-similar form
\begin{equation}\label{C22'}
u(x,t)=t^{\alpha/(p+\alpha(2-p))}f(\xi),\quad \xi =xt^{-1/(p+\alpha (2-p))},
\end{equation}
where $f$ satisfies \eqref{C2'}. Moreover, we have 
\begin{equation}\label{C23'}
Dt^{\alpha/(p+\alpha(2-p))}(\xi_3+\xi)^{p/(p-2)}\leq u\leq C_7t^{\alpha/(p+\alpha(2-p))}(\xi_4+\xi)^{p/(p-2)},\qquad 0\leq x, t<+\infty.
\end{equation}
(see Appendix). The right-hand side of \eqref{C23'} is not sharp enough as  $x\rightarrow +\infty$ and the required upper estimation is provided by an explicit solution to \eqref{C1}, as in  \eqref{C16'}. From \eqref{C23'} and \eqref{C16'} it follows that, for arbitrary fixed $0<t<+\infty$, the asymptotic result \eqref{C17'} is valid.
Now assume that $u_0$ satisfies \eqref{C3} with $\alpha>0$. Then \eqref{C1'} is valid and for an arbitrary sufficiently small $\epsilon>0$ there exists a $\delta=\delta(\epsilon)>0$ such that the estimation \eqref{C23'} is valid for $0<t\leq \delta$, except that in the left-hand side (respectively in the right-hand side ) of \eqref{C23'} the constant $A_0$ should be replaced by $A_0-\epsilon$ (respectively $A_0+\epsilon$). Moreover, there exists a number $\delta>0$ (which does not depend on $\epsilon$) such that, for arbitrary $t\in(0,\delta]$, the asymptotic result \eqref{C17'} is valid. 
\section{Preliminary Results}\label{sec: preliminary results} 
%%%% Most of the enunciations like theorem, lemma, corollary, proposition, defintion,
%%%% condition, example, conjecture etc. are defined in the class file.

%%%% If the author wants to add or modify the enunciation style
%%%% they can define in the preamble as shown below.

%%%% \newtheoremstyle{theorem}{6pt}{6pt}{\rm}{}{\sffamily}{ }{ }{}
%%%% \theoremstyle{theorem}
%%%% \newtheorem{theorem}{\sc Theorem}[section]

%%%%\newtheoremstyle{corollary}{6pt}{6pt}{\rm}{}{\sffamily}{ }{ }{}
%%%%\theoremstyle{corollary}
%%%%\newtheorem{corollary}{\sc Corollary}[section]

%%%%\newtheoremstyle{definition}{6pt}{6pt}{\rm}{}{\sffamily}{ }{ }{}
%%%%\theoremstyle{definition}
%%%%\newtheorem{definition}[theorem]{\sc Definition}
%%%%
%%%%\newtheorem{exercise}[theorem]{Exercise}

The following is the standard comparison result (\cite{dibe-sv})
\begin{lemma}\label{lemma1}
Let either $b\geq 0$ or $b<0$ and $\beta\geq 1$. Assume that $g$ be a nonnegative and continuous function in $\overline{Q}$, where
\[ Q=\{(x,t): \eta_0(t)<x<+\infty, 0< t <T\leq +\infty\},  \]
$f$ is in $C_{x,t}^{2,1}$ in $Q$ outside a finite number of curves $x=\eta_j(t)$, which divide $Q$ into a finite number of subdomains $Q^j$, where $\eta_j \in C[0,T]$; for arbitrary $\delta >0$ and finite $\delta_1\in (\delta, T]$ the function $\eta_j$ is absolutely continuous in $[\delta, \delta_1]$. Let $g$ satisfy the inequality 
\[  Lg\equiv g_t-\Big(|g_x|^{p-2}g_x\Big)_x+bg^{\beta} \geq 0, \ (\leq 0)  \]
at the points of $Q$, where $g \in C_{x,t}^{2,1}$. Assume also that the function $|g_x|^{p-2}g_x$ is continuous in $Q$ and $g\in L^{\infty}(Q\cap (t\leq T_1))$ for any finite $T_1\in (0,T]$. Then $g$ is a supersolution (subsolution) of \eqref{C1}. If, in addition we have
\[ g\Big |_{x=\eta_0(t)} \geq (\leq) \ u\Big |_{x=\eta_0(t)}, \ g\Big |_{t=0} \geq (\leq) \ u\Big |_{t=0} \]
then
\[ g \geq (\leq) \ u, \quad\text{in} \ \   \overline{Q}  \]
\end{lemma}

In the next two lemmas, we establish some preliminary estimations of the solution to CP, the proof of these estimations being based on scale of variables.
\begin{lemma}\label{lemma3.1}  If $b=0$ and $1<p<2,\; \alpha >0,$ then the solution $u$ of the CP \eqref{C1}, \eqref{C4} has the self-similar form \eqref{C22'}, where the self-similarity function $f$ satisfies \eqref{C2'}. If $u_0$ satisfies \eqref{C3} then the solution to the CP \eqref{C1}, \eqref{C2} satisfies \eqref{C1'}. \end{lemma}
The proof of the lemma coincides with the proof of Lemma 5 from \cite{Abdulla35}.
\begin{lemma}\label{lemma3.2} Let $u$ be a solution of the \eqref{C1}, \eqref{C2} and let $u_0$ satisfy  \eqref{C3}. Let one of the following conditions be valid:
\begin{itemize}%[label=(\alph*)]
\item (a) \ $ b>0,~\; 0<\beta<p-1<1,~\; 0<\alpha <\frac{p}{p-1-\beta}$;
\item (b) \ $b>0,~\; 0<p-1<1,~\; \beta\geq p-1,~\; \alpha>0$;
\item (c) \ $b<0,~\; \beta \geq 1,~\; 0<p-1 <1,~\; \alpha >0$.
\end{itemize}
Then $u$ satisfies \eqref{C1'} with the same function $f$ as in Lemma \ref{lemma3.1}. \end{lemma}

\begin{lemma}\label{lemma3.3} Let $u$ be a solution to the CP  \eqref{C1}, \eqref{C4} with $b>0,\; 0<\beta<1,\; p-1>\beta,\;\alpha=p/(p-1-\beta),\;C>0$. Then the solution $u$ has the self-similar form  \eqref{C4'}. There is a constant $\ell_1>0$ such that for arbitrary $\ell\in(-\infty,-\ell_1]$ there exists $\lambda>0$ such that  \eqref{C7'} is valid. If $0<C<C_*$ then 
 \begin{equation}\label{C7''}
0<\lambda<C_*(-\ell)^{p/(p-1-\beta)}.\end{equation}
 If $C>C_*$ ~ then ~$f_1(0)=A_1>0$ where $A_1$ depends on $p,\;\beta,\; C$ and $b$.\end{lemma}
 
\begin{lemma}\label{lemma3.4} Let $u$ be a solution to the CP   \eqref{C1}-\eqref{C3} with $b>0,\;0<\beta <1,\;p-1>\beta,\; \alpha=p/(p-1-\beta),\; C>0.$ Then for  arbitrary $\ell\in(-\infty,-\ell_1]$ we have 
 \begin{equation}\label{C12''}
u(\ell t^{(p-1-\beta)/(p(1-\beta))},t)\sim \lambda t^{1/(1-\beta)} ~\text{as}~t\rightarrow 0^+,
\end{equation}
 where $\ell_1>0,\; \lambda >0$ are the same as in Lemma \ref{lemma3.3} and if $0<C<C_*$ then  \eqref{C7''} is also valid. If $C>C_*$ then $u$ satisfies 
 \begin{equation}\label{C13''}
u(0,t)\sim A_1 t^{1/(1-\beta)}~\text{as}~t\rightarrow 0^+,
\end{equation}
 where $A_1=f_1(0)>0$ (see Lemma \ref{lemma3.3}).\end{lemma}
\begin{lemma}\label{lemma3.5} Let $u$ be a solution to the CP \eqref{C1}-\eqref{C3} with $b>0,\; 0<\beta<1,\; p-1>\beta,\; \alpha>p/(p-1-\beta),\; C>0.$ Then for arbitrary $\ell>\ell_*$ (see \eqref{C10'}), the asymptotic formula \eqref{C11'} is valid.\end{lemma}
 \begin{proof}[Proof of Lemma~\ref{lemma3.2}] 
 The proof for cases (a) and (b) coincides with the proof for case (a) and (b) with $b>0$ in Lemma \ref{lemma3.2} of \cite {Abdulla35}. The proof  for (c) coincides (with some modifications) with the proof for case (b) with $b<0$ in the Lemma 6 for \cite{Abdulla35}; namely, instead of zero boundary condition on the line $x=-x_{\epsilon}$ and $x=-k^{1/\alpha}x_{\epsilon}$ (see (3.17) and (3.18) in \cite{Abdulla35}) we take 
 \[u_{\pm \epsilon}(-x_{\epsilon},t)=u(-x_{\epsilon},t),~0\leq t\leq \delta,\]
 \[u_{k}^{\pm \epsilon}(-k^{\frac{1}{\alpha}}x_{\epsilon},t)=ku(-x_{\epsilon},k^{(\alpha(p-2)-p)/\alpha}t),~\; 0\leq t \leq k^{\frac{p-\alpha(p-2)}{\alpha}}\delta,\]
 which are used to imply (3.9) from \cite{Abdulla35}. Moreover, if $\beta>1$ then to prove uniform boundedness of the sequence $\{u_{k}^{\pm \epsilon}\}$ we choose 
 \[g(x,t)=(C+1) (1+x^2)^{\frac{\alpha}{2}}(1-\nu t)^{\frac{1}{1-\beta}},~x\in \mathbb{R},~0\leq t\leq t_0=\frac{\nu ^{-1}}{2},\]
 where $\nu,\; h_*$ are chosen as in \cite{Abdulla35} and 
 \[h(x)=(\beta-1)\alpha^{p-1}(C+1)^{p-2}(1+x^2)^{\frac{(\alpha-2)(p-1)-2-\alpha}{2}}x^{2}|x|^{p-2}\Big(\frac{1+x^2}{x^2}+(p-2)\frac{1+x^2}{|x|^{2}}+(\alpha-2)(p-1)\Big)\]
 Then, we have
\[L_{kg}\equiv g_t-\big(|g_x|^{p-2}g_x\big)_x+bk^{\frac{\alpha(p-\beta -1)-p}{\alpha}}g^{\beta}=(C+1)(\beta-1)^{-1}(1+x^2)^{\frac{\alpha}{2}}(1-\nu t)^{\frac{\beta}{1-\beta}}S,\]
where
\[S=\nu-h(x)+b(\beta-1)(C+1)^{\beta -1}k^{\frac{\alpha(p-\beta -1)-p}{\alpha}}(1+x^2)^{\frac{\alpha(\beta-1)}{2}}.\]
Let $R=b(\beta-1)(C+1)^{\beta -1}k^{\frac{\alpha(p-\beta -1)-p}{\alpha}}(1+x^2)^{\frac{\alpha(\beta-1)}{2}}.$
\begin{equation}\label{CP17''}S\geq 1+R~\text{in}~D_{0\epsilon}^{k}=D_{\epsilon}^{k}\cap \{0<t\leq t_{0}\},\end{equation}
where
\[R=O\Big(k^{p-2-p/\alpha}\Big)~~~\text{uniformly for}~~(x,t)\in D_{0\epsilon}^{k}~~\text{as}~~k\rightarrow +\infty.\]
 while if $\beta=1$ we take
 \[g=(C+1) \exp(\nu t)(1+x^2)^{\frac{\alpha}{2}}\]
 where
 \[\nu=1+\max_{x\in\mathbb{R}}h(x).\]
 \[h(x)=(C+1)^{p-2}\alpha^{p-1}(1+x^2)^{\frac{(\alpha-2)(p-1)-2-\alpha}{2}}x^{2}|x|^{p-2}\Big(\frac{1+x^2}{x^2}+(p-2)\frac{1+x^2}{|x|^{2}}+(\alpha-2)(p-1)\Big).\]
  Then, we have
 \[L_{kg}\equiv g_t-\big(|g_x|^{p-2}g_x\big)_x+bk^{\frac{\alpha(p-2)-p}{\alpha}}g=(C+1)(1+x^2)^{\frac{\alpha}{2}}\exp(\nu t)S,\]
 where \[S=\nu-h(x)+bk^{\frac{\alpha(p-2)-p}{\alpha}}.\]
 Let $R=bk^{\frac{\alpha(p-2)-p}{\alpha}}$. Since $\alpha(p-2)-p<0$, then $R\rightarrow 0$ as $k\rightarrow +\infty$. 
\[R=O\Big(k^{\frac{\alpha(p-2)-p}{\alpha}}\Big)~~~\text{uniformly for}~~(x,t)\in D_{0\epsilon}^{k}~~\text{as}~~k\rightarrow +\infty.\]
Moreover, we have for~~$0<\epsilon \ll 1$
\begin{subequations}
\begin{equation}\label{CP18a''} g(x,0)~~~~~\geq u_k^{\pm \epsilon}(x,0)~\text{for}~|x|\leq k^{1/\alpha}|x_{\epsilon}|,\end{equation}
\begin{equation}\label{CP18b''}g(\pm k^{1/\alpha}x_{\epsilon},t)\geq u_k^{\pm \epsilon}(\pm k^{1/\alpha}x_\epsilon,t)~\text{for}~0\leq t\leq t_0.\end{equation}
\end{subequations}
Hence, $\exists \ k_0=k_0(\alpha;p)$~such that for $\forall k\geq k_0$~the comparison theorem implies
\begin{equation}\label{CP19''}0 \leq u_k^{\pm \epsilon}(x,t)\leq g(x,t)~\text{in}~\bar D_{0\epsilon}^k.\end{equation}
Let~$ G$~ be an arbitrary fixed compact subset of 
\[P=\big\{(x,t):x\in \mathbb{R},~~~0<t\leq t_0\big\}.\]
We take~$k_0$~so large that~$G\subset D_{0\epsilon}^k$~for~$k\geq k_0.$~From \eqref{CP19''}, it follows that the sequences~$\{u_k^{\pm \epsilon}\}$,~$k\geq k_0 $,~are uniformly bounded in~$G$.~As before, from the results of \cite{dibe-sv, tsutsumi} it follows that the sequence of non-negative and locally bounded solutions $\{u_{k}^{\pm \epsilon}\}$ is locally uniformly H\"{o}lder continuous, and weakly pre-compact in $W_{loc}^{1,p}(\mathbb{R}\times(0,T))$. It follows that for some subsequence~$k'$~
\begin{equation}\label{CP20''}\underset{k'\rightarrow+\infty}\lim u_{k'}^{\pm \epsilon}(x,t)=v_{\pm \epsilon}(x,t),~~~~~~~(x,t)\in P.\end{equation}
Since $\alpha(p-1-\beta)-p<0,$ passing to limit as $k'\to +\infty$, from \eqref{supersub} for $u_{k'}^{\pm \epsilon}$ it follows that~$v_{\pm \epsilon}$~is a solution to the CP \eqref{C1}, \eqref{C2} with~$b=0, T=t_0, u_0=(C\pm \epsilon)(-x)_+^{\alpha}.$~From Lemma~\ref{lemma3.1},  the required estimation \eqref{C2'} follows.
\end{proof}

 \begin{proof}[Proof of Lemma~\ref{lemma3.3}] 
 The first assertion of the lemma is known when $p-1\geq 1$ (see Lemma \ref{lemma3.3} of \cite{Abdulla35}). The proof is similar if $\beta <p-1<1$. If we consider a function
  \begin{equation}\label{C8''}
 u_k(x,t)=ku(k^{-\frac{p-1-\beta}{p}}x,k^{\beta-1}t),\qquad k>0,
 \end{equation}
 It may easily be checked that this satisfies   \eqref{C1}, \eqref{C4}. Since under the conditions of the lemma there exists a unique global solution to   \eqref{C1}, \eqref{C4} we have 
 \begin{equation}\label{C9''}
u(x,t)=ku(k^{-\frac{p-1-\beta}{p}}x,k^{\beta-1}t),\qquad k>0.\end{equation}
 If we choose $k=t^{1/(1-\beta)}$ then  \eqref{C9''} implies then  \eqref{C4'} with $f_1(\zeta)=u(\zeta,1).$\\[.2cm]
 To prove the second assertion of the lemma, take an arbitrary $x_1<0$. Since $u$ is continuous, there exists $\delta_1>0$ such that 
\begin{subequations}
\begin{equation}\label{C10''a}
(C/2)(-x_1)^{p/(p-1-\beta)}\leq u(x_1,\delta) ~~\text{for}~ \delta \in [0,\delta_1]
\end{equation}
If $C\in (0,C_*)$ then we also choose $\delta_1>0$ such that 
\begin{equation}\label{C10''b}
u(x_1,\delta)<C_*(-x_1)^{p/(p-1-\beta)} ~~\text{for}~ \delta \in [0,\delta_1]
\end{equation}
\end{subequations}
 Choose $k=(t/\delta)^{1/(1-\beta)}$ in \eqref{C9''} and then taking 
 \[x=-\ell t^{(p-1-\beta)/p(1-\beta)},\qquad \ell=\ell(\delta)=x_1\delta^{-(p-1-\beta)/p(1-\beta)}, \; \delta \in (0,\delta_1]\]
 we obtain \eqref{C7'} with 
 \[ \ell_1= -x_1 \delta_1^{-(p-1-\beta)/p(1-\beta)},\qquad \lambda=\lambda(\delta)=\delta^{1/(\beta-1)}u(x_1,\delta), \; \delta \in (0,\delta_1].\]
 If $0<C<C_*,$ then  \eqref{C7''} follows from \eqref{C10''b}. Let $C>C_*,$ to prove that $f_1(0)=A_1>0$ it is enough to prove that there exists a $t_0>0$ such that 
\begin{equation}\label{C11''}
u(0,t_0)>0.
\end{equation}
 If $p\geq 2$, \eqref{C11''} is a known result (see Lemma \ref{lemma3.3} of \cite{Abdulla35}). To prove \eqref{C11''} when $\beta<p-1<1$,
Consider the function
\[g(x,t)=C_1(-x+t)_+^{p/(p-1-\beta)}\]
where $C_1\in(C_*,C)$. If $x<t$ we have 
 \[Lg=bg^{\beta}S,\qquad S=1-\big(\frac{C_1}{C_*}\big)^{p-1-\beta}+\frac{p}{b(p-1-\beta)}C_1^{1-\beta}(-x+t)^{(\beta(1-p)+1)/(p-1-\beta)}.\]
 we can choose $x_1<0$ and $t_1>0$ such that 
 \[S\leq0 \qquad \text{if}~~ x_1\leq x \leq t,\;~~0\leq t \leq t_1.\]
 Since $u$ is continuous, we can also choose $t_1>0$ sufficiently small that
\[g(x_1,t)\leq u(x_1,t) ~~\text{for}~~0\leq t\leq t_1.\]
Moreover
 \[g(x,0)\leq u_0(x) ~~\text{for}~~x\geq x_1\]
 Applying comparison Lemma~\ref{lemma1} we have 
 
 \[u(x,t)\geq g(x,t) \qquad\text{for}~~x\geq x_1,\;0\leq t\leq t_1,\]
 which implies \eqref{C11''}. The lemma is proved.\end{proof}
 
Lemma \ref{lemma3.4} may be proved by localization of the proof given in Lemma \ref{lemma3.3}.%, exactly as similar local results were proved in Lemma 5 of \cite{Abdulla35}.
 The proof of Lemma \ref{lemma3.5} coincides with the proof of Lemma 8 from \cite{Abdulla35}.

\section{Proofs of the Main Results}\label{sec: proofs of the main results.}
\begin{proof}[Proof of Theorem~\ref{diffusiondominates'}]
The asymptotic estimations \eqref{C1'} and \eqref{C2'} follow from Lemma \ref{lemma3.2}. Take an arbitrary sufficiently small number $\epsilon>0$; from \eqref{C1'} it follows that there exists a number $\delta_1=\delta_1(\epsilon)>0$ such that
 \begin{equation}\label{C1'''}
(A_0-\epsilon)t^{\frac{\alpha}{p-\alpha (p-2)}}\leq u(0,t)\leq (A_0+\epsilon)t^{\frac{\alpha}{p-\alpha (p-2)}},~~~~0\leq t\leq \delta_1,
\end{equation}
where $A_0=f(0)>0.$ Consider a function 
\begin{equation}\label{C2'''}
g(x,t)=t^{1/(1-\beta)}f_1(\zeta), ~~~\zeta=xt^{-(p-1-\beta)/p(1-\beta)}.
\end{equation}
We have 
\begin{subequations}\label{C3'''ab}
\begin{equation}\label{C3'''a}
Lg=t^{\frac{\beta}{1-\beta}}\mathcal{L}^0f_1,
\end{equation}
\begin{equation}\label{C3'''b}
\mathcal{L}^0f_1=\frac{1}{1-\beta}f_1(\zeta)-\frac{p-1-\beta}{p(1-\beta)}\zeta f'_1(\zeta)-\big(|f'_1(\zeta)|^{p-2}f'_1(\zeta)\big)'+bf_1^{\beta}.
\end{equation}
 \end{subequations}
For the function $f_1$, we take 
\[f_1(\zeta)=C_0(\zeta_0-\zeta)_+^{p/(p-1-\beta)},~0<\zeta<+\infty\]
where $~C_0,\zeta_0~$are some positive constants. From \eqref{C3'''b}, we then have 
\begin{equation}\label{C4'''}
\mathcal{L}^0f_1=bC_0^{\beta}(\zeta_0-\zeta)_+^{\frac{p\beta}{p-1-\beta}}\Big\{1-\Big(\frac{C_0}{C_*}\Big)^{p-1-\beta}+ \frac{C_0^{1-\beta}}{b(1-\beta)}\zeta_0(\zeta_0-\zeta)_+^{\frac{\beta(1-p)+1}{p-1-\beta}}\Big\}.
\end{equation}
To prove a lower estimation, we take  $C_0=C_1,\zeta_0=\zeta_1$ (see Appendix). Then we have 
\begin{equation}\label{C5'''}
\mathcal{L}^0f_1\leq bC_1^{\beta}(\zeta_1-\zeta)_+^{\frac{p\beta}{p-1-\beta}}\Big\{1-\Big(\frac{C_1}{C_*}\Big)^{p-1-\beta}+ \frac{C_1^{1-\beta}}{b(1-\beta)}\zeta_1^{\frac{p(1-\beta)}{p-1-\beta}}\Big\}=0.
\end{equation}
From  \eqref{C3'''ab}, it follows that 
\begin{subequations}\label{C6'''ab}
\begin{equation}\label{C6'''a}
Lg\leq 0\quad \text{for}~0<x<\zeta_1t^{\frac{p-1-\beta}{p(1-\beta)}},\; 0<t<+\infty,
\end{equation}
 \begin{equation}\label{C6'''b}
Lg=0\quad \text{for}~x>\zeta_1t^{\frac{p-1-\beta}{p(1-\beta)}},\; 0<t<+\infty.
\end{equation}
 \end{subequations}
Lemma~\ref{lemma1} implies that $g$ is a subsolution of \eqref{C1} in $\{(x,t):x>0, t>0\}.$
Since $1/(1-\beta)>\alpha/(p-\alpha(p-2)),$ it follows from \eqref{C1'''} that there exists a $\delta_2>0,$ which does not depend on $\epsilon,$ such that  
\begin{subequations}\label{C7'''ab}
 \begin{equation}\label{C7'''a}
 g(0,t)\leq u(0,t)\quad \text{for}\; 0\leq t\leq\delta_2.
 \end{equation}
We also have  
\begin{equation}\label{C7'''b}
 g(x,0)=u(x,0)=0\quad \text{for}\; 0\leq x<+\infty.
\end{equation} \end{subequations}
Now we can fix a particular value of $\epsilon=\epsilon_0$ and take $\delta=\min(\delta_1,\delta_2).$ 
From \eqref{C6'''ab}, \eqref{C7'''ab} and Lemma~\ref{lemma1}, the left-hand sides of \eqref{C3'a}, \eqref{C3'b} follow. To prove an upper estimation, we first use the rough estimation \eqref{C16'}.
%\begin{equation}\label{C8'''}
%u(x,t)\leq D t^{1/(2-p)}x^{p/(p-2)},~\; 0<x<+\infty,\; 0<t<+\infty.
%\end{equation}
The estimation \eqref{C16'} is obvious, since by the comparison theorem, $u(x,t)$ may be upper estimated by the solution of \eqref{C1} with $b=0$. Using \eqref{C16'}, we can now establish a more accurate estimation. For that, consider a function $g$ with $C_0=C_*,\;\zeta_0=\zeta_2$ in $G_{\ell_0,\delta}$, where
\[G_{\ell_0,\delta}=\{(x,t):\zeta_{\ell_0}(t)=\ell_0 t^{(p-1-\beta)/p(1-\beta)}<x<+\infty,\; 0<t\leq \delta\}.\]
From \eqref{C3'''ab},\eqref{C4'''} it follows that
 \begin{subequations}\label{C9'''ab}
\begin{equation}\label{C9'''a}
Lg\geq 0\quad \text{for}~0<x<\zeta_2t^{\frac{p-1-\beta}{p(1-\beta)}},\; 0<t<+\infty,
\end{equation}
 \begin{equation}\label{C9'''b}
Lg=0\quad \text{for}~x>\zeta_2t^{\frac{p-1-\beta}{p(1-\beta)}},\; 0<t<+\infty.
\end{equation}
 \end{subequations}
Moreover, from \eqref{C16'} we have
\begin{equation}\label{C10'''}
u(\zeta_{\ell_0}(t),t)\leq D \ell_0^{p/(p-2)}t^{1/(1-\beta)}=C_*(\zeta_2-\ell_0)_+^{p/(p-1-\beta)}t^{1/(1-\beta)}=g(\zeta_{\ell_0}(t),t)\quad \text{for}~ 0\leq t\leq\delta.\end{equation}
By applying Lemma~\ref{lemma1} in  $G_{\ell_0,\delta}$, the right hand side of  \eqref{C3'b} follows from \eqref{C9'''ab},  \eqref{C10'''}  and  \eqref{C7'''b}. 

If $u_0$ is defined as in \eqref{C4}, then the CP \eqref{C1}, \eqref{C4} has a global solution and from comparison theorem it follows that the solution may be globally upper estimated by the solution to the CP \eqref{C1}, \eqref{C4} with $b=0$. Hence \eqref{C10'''} and the right-hand side of  \eqref{C3'a} are valid for $0<t<+\infty.$ \end{proof}
%%%%%%%%%%%%%%%%%%%%%%%%%%%%%%%%%%%%%%%%%%%%%%%%%%%%%%%%%%

\begin{proof}[Proof of Theorem~\ref{diffusionbalance'}]
First, assume that $u_0$ is defined by \eqref{C4}. 
The self-similar form \eqref{C4'} follows from  Lemma \ref{lemma3.3}. The proof  of the estimation  \eqref{C6'a} when $C>C_*$ and the proof of the right-hand side of \eqref{C6'a} 
when $0<C<C_*$ (and of the corresponding local ones when $u_0$ satisfies \eqref{C3}) fully coincides with the proof given in \cite{Abdulla35} for the case $1<(p-1)<\beta^{-1}$ (see (2.16) and (2.19) in \cite{Abdulla35}). To prove the left-hand side of \eqref{C6'a}, consider a function $g$ from \eqref{C2'''} with
\[f_1(\zeta)=C_*(-\zeta_5-\zeta)_+^{p/(p-1-\beta)},\qquad -\infty<\zeta<+\infty\]
From \eqref{C3'''ab},\eqref{C4'''} it follows that
\begin{equation}\label{C11'''}
Lg\leq 0\quad \text{in}~G_{-\ell_1,\infty}
\end{equation}
Moreover, we have
\begin{subequations}\label{C12'''ab}
\begin{equation*}
u(-\ell_1 t^{(p-1-\beta)/(p(1-\beta))},t)= \lambda t^{1/(1-\beta)}=C_*(\ell_1-\zeta_5)_+^{p/(p-1-\beta)}t^{1/(1-\beta)}\end{equation*}
\begin{equation}\label{C12'''a}
 =g(-\ell_1 t^{(p-1-\beta)/p(1-\beta)},t),~\qquad 0\leq t<+\infty,
\end{equation}
\begin{equation}\label{C12'''b}
 u(x,0)=g(x,0)=0,~\qquad 0\leq x\leq x_0,
\end{equation}
\begin{equation}\label{C12'''c}
u(x_0,t)=g(x_0,t)=0,~\qquad 0\leq t<+\infty,
\end{equation}
\end{subequations}
where $x_0>0$ is an arbitrary fixed number. By using \eqref{C11'''} and \eqref{C12'''ab}, we can apply Lemma~\ref{lemma1} in
\[G'_{-\ell_1,\infty}=G_{-\ell_1,\infty}\cap \{x<x_0\}.\]
Since $x_0>0$ is an arbitrary number, the desired lower estimation from \eqref{C6'a} follows .

Suppose that $u_0$ satisfies \eqref{C3} with $\alpha=p/(p-1-\beta),\; 0<C<C_*.$ Then from \eqref{C12''}, it follows that for an arbitrary sufficiently small $\epsilon>0$ there exists a number $\delta=\delta(\epsilon)>0$ such that 
\begin{equation*}
 (\lambda-\epsilon) t^{1/(1-\beta)}\leq u(-\ell_1 t^{(p-1-\beta)/(p(1-\beta))},t)\leq (\lambda+\epsilon) t^{1/(1-\beta)},\qquad 0\leq t\leq \delta.\end{equation*}
Using this estimation, the left-hand side of \eqref{C6'a} may be established locally in time. The proof completely coincides with the proof given above for the global estimations, except that $\lambda$ should be replaced by $\lambda-\epsilon.$ \eqref{C4'} and \eqref{C6'a} easily imply \eqref{C5'} and \eqref{C6'b}.
\end{proof}

\begin{proof}[Proof of Theorem~\ref{reactiondominates'}]
The asymptotic estimation  \eqref{C11'} follows from Lemma \ref{lemma3.5}. The proof of the asymptotic estimation \eqref{C10'} coincides with the proof given in \cite{Abdulla35}. In particular, the estimations (4.19) and (4.20) from \cite{Abdulla35} are true in this case as well. \end{proof}

\begin{proof}[Proof of Theorem~\ref{infinitespeed}]
The asymptotic estimations \eqref{C1'} and  \eqref{C2'} follow from Lemma \ref{lemma3.2}. From \eqref{C1'}, \eqref{C1'''} follows, where we fix a particular value of $\epsilon=\epsilon_0$. %For arbitrary given $M>0$, the function $g_M(x,t)=t^{1/(2-p)}\phi_M(x)$ is a solution of \eqref{C1}. Since $1/(2-p)> \alpha/(p+\alpha(2-p))$, there exists $\delta=\delta(M)>0$ such that 
The function $g(x,t)=t^{1/(2-p)}\phi(x)$ is a solution of \eqref{C1}. Since $1/(2-p)> \alpha/(p+\alpha(2-p))$, there exists $\delta>0$ such that 
 %\[u(0,t)=A_0t^{\frac{\alpha}{p+\alpha (2-p)}}\geq M t^{\frac{1}{2-p}}= \phi_M(0) t^{\frac{1}{2-p}}=g_M(0,t),\qquad 0\leq t \leq \delta.\]
 \[u(0,t)=A_0t^{\frac{\alpha}{p+\alpha (2-p)}}\geq t^{\frac{1}{2-p}}= \phi(0) t^{\frac{1}{2-p}}=g(0,t),\qquad 0\leq t \leq \delta.\]
% \[u(x,0)=g_M(x,0)=0,\qquad 0\leq x<\infty\]
 \[u(x,0)=g(x,0)=0,\qquad 0\leq x<\infty\]
 Therefore, from Lemma~\ref{lemma1}, the left-hand side of \eqref{C12'} follows. Let us prove the right-hand side of \eqref{C12'}. As it was mentioned in Section \ref{sec: description of the main results}, the right-hand side of \eqref{C12'} is valid for $0<t<+\infty$ if the initial data $u_0$ from \eqref{C2} vanishes for $x\geq 0$. %Take an arbitrary $M>0$ and consider a function 
For all $\epsilon>0$ and consider a function 
 %\[g^M(x,t)=(t+M^{-\gamma})^{1/(2-p)}\phi_M(x), \qquad \gamma\in (0,2-p)\]
  \[g_\epsilon(x,t)=(t+\epsilon)^{1/(2-p)}\phi(x),\]
 %\[g^M(0,t)=(t+M^{-\gamma})^{1/(2-p)}\phi_M(0)=(t+M^{-\gamma})^{1/(2-p)}M\geq M^{1-\gamma/(2-p)}\geq\]
  %\[\geq A_0 t^{\frac{\alpha}{p+\alpha (2-p)}}=u(0,t),\quad \text{for}\;0\leq t\leq T(M)= \big[A_0^{-1} M^{1-\gamma/(2-p)}\big]^{\frac{p+\alpha (2-p)}{\alpha}},\]
   \[g_\epsilon(0,t)=(t+\epsilon)^{1/(2-p)}\phi(0)=(t+\epsilon)^{1/(2-p)}\geq\epsilon^{1/(2-p)}\geq\]
   \[\geq \big(A_0+\epsilon\big) t^{\frac{\alpha}{p+\alpha (2-p)}}=u(0,t),\;\;\text{for}\;0\leq t\leq \delta_\epsilon= \big[(A_0+\epsilon\big)^{-1}\epsilon^{1/(2-p)}\big]^{\frac{p+\alpha (2-p)}{\alpha}},\]
  Due to continuity of $g_\epsilon$ and $u$, $\exists~\delta_{1\epsilon}>0$ such that $g_\epsilon(0,t)\geq u(0,t)$.
 %where $T(M)\rightarrow +\infty $ as $M\rightarrow +\infty$. Since $g^M$ is a solution of Eq. \eqref{C1}, from the Comparison Theorem 2.4 of \cite{Abdulla3} it follows that
  Since $g_\epsilon$ is a solution of \eqref{C1}, from the Lemma~\ref{lemma1} it follows that
   %\begin{equation}\label{a'bC13'}u(x,t)\leq g^M(x,t)=(t+M^{-\gamma})^{1/(2-p)}\phi_M(x),\quad \text{for}\;0\leq x<+\infty,\;0\leq t\leq T(M).\end{equation}
      \begin{equation}\label{a'bC13'}u(x,t)\leq g_\epsilon(x,t)=(t+\epsilon)^{1/(2-p)}\phi(x),\quad \text{for}\;0\leq x<+\infty,\;0\leq t\leq \delta_\epsilon.\end{equation}
Integration of \eqref{C13'ab} implies \eqref{integral1}. Global estimation \eqref{phiglobal} \eqref{phiglobal1} \eqref{phiglobal2} \eqref{phiglobal3}  
By rescaling $x\to \epsilon^{-1}x, \epsilon>0$ from  \eqref{integral1} we have
\[\frac{x}{\epsilon}=\int_{\phi(\frac{x}{\epsilon})}^{1}y^{-1}\big[\frac{b}{p-1}+\frac{p}{2(p-1)(2-p)}y^{2-p}\big]^{-1/p}dy.\]
Change of variable $z=-\epsilon \log y$ implies
\begin{equation}\label{Fepsilon}
x=\mathcal{F}[\Phi_\epsilon(x)],
\end{equation}
where
%\[\frac{x}{\epsilon}=-\frac{1}{\epsilon}\int_{-\epsilon \log\phi(\frac{x}{\epsilon})}^{0}e^{\frac{z}{\epsilon}}\big[\frac{b}{p-1}+\frac{p}{2(p-1)(2-p)}e^{\frac{(p-2)}{\epsilon}z}\big]^{-1/p}e^{-\frac{z}{\epsilon}}dz\]
%$\Rightarrow$
\[\mathcal{F}(y)=\int^{y}_{0}\big[\frac{b}{p-1}+\frac{p}{2(p-1)(2-p)}e^{\frac{(p-2)}{\epsilon}z}\big]^{-1/p}dz, \]
\[ \Phi_\epsilon(x)=-\epsilon \log\phi(\frac{x}{\epsilon}).\]
From \eqref{Fepsilon} it follows that
\begin{equation}\label{Fepsilon1}
\Phi_\epsilon(x)=\mathcal{F}^{-1}(x),
\end{equation}
where $\mathcal{F}^{-1}$ is an inverse function of $\mathcal{F}$. Since $1<p<2$ it easily follows that
\begin{equation}\label{Fepsilon2} 
\lim_{\epsilon \rightarrow 0} \mathcal{F}(y)=\Big(\frac{b}{p-1}\Big)^{-1/p}y, \ \lim_{\epsilon \rightarrow 0} \mathcal{F}^{-1}(y)=\Big(\frac{b}{p-1}\Big)^{1/p}y,
\end{equation}
for $y\geq 0$ and convergence is uniform in bounded subsets of $\mathbb{R}^+$. From \eqref{Fepsilon1}, \eqref{Fepsilon2} it follows that
\begin{equation}\label{Fepsilon3}
-\lim_{\epsilon \rightarrow 0}\epsilon \log\phi\Big(\frac{x}{\epsilon}\Big)=\Big(\frac{b}{p-1}\Big)^{1/p}x,\quad 0<x<+\infty.
\end{equation}
By letting $y=x/\epsilon$ from \eqref{Fepsilon3}, \eqref{logasymp'.} follows. Global estimation \eqref{phiglobal}, and accordingly also \eqref{phiglobal2} \eqref{phiglobal3} easily follow from
 \eqref{integral1}, \eqref{integral2}.
%%%%%%%%%%%%%%%%%%%
%%%%%%%%%%%%%%%%%%%%%%%%%%%%%%%%%%%%%%%%%%%%%%
\end{proof}

\begin{proof}[Proof of Theorem~\ref{infinitespeed'}]
Let either $b>0,\; \beta>p-1$  or $b<0,\;\beta \geq 1$. The asymptotic estimations \eqref{C1'} and  \eqref{C2'} follow from Lemma \ref{lemma3.2}. Take an arbitrary sufficiently small number $\epsilon>0$. From \eqref{C1'}, it follows that there exists a number $\delta_1=\delta_1(\epsilon)>0$ such that  \eqref{C1'''} is valid. Let $\beta \geq 1$.

 Consider a function
\begin{equation}\label{C13'''}
g(x,t)=t^{\alpha/(p+\alpha(2-p))}f(\xi), ~~~\xi=xt^{-1/(p+\alpha(2-p))}.
\end{equation}
We have 
\begin{subequations}\label{C14'''ab}
\begin{equation}\label{C14'''a}
Lg=t^{(\alpha(p-1)-p)/(p+\alpha(2-p))}\mathcal{L}_tf\end{equation}
\begin{equation}\label{C14'''b}
\mathcal{L}_tf=\frac{\alpha}{p+\alpha(2-p)}f-\frac{1}{p+\alpha(2-p)}\xi f'-\big(|f'|^{p-2}f'\big)'+bt^{(p-\alpha(p-1-\beta))/(p-\alpha(p-2))}f^{\beta}.\end{equation}
\end{subequations}
As a function $f$ we take 
\begin{equation}\label{C15'''}
f(\xi)=C_0(\xi_0+\xi)^{-\gamma_0},\qquad 0\leq \xi<+\infty\end{equation}
where $C_0$,\;$\xi_0$,\;$\gamma_0$ are some positive constants. Taking $\gamma_0=p/(2-p)$ from \eqref{C14'''b} we have
\begin{subequations}\label{C16'''ab}
\begin{equation*}
\mathcal{L}_tf=(p+\alpha(2-p))^{-1}C_0(\xi_0+\xi)^{\frac{p}{p-2}}\end{equation*}
\begin{equation}\label{C16'''a}
\times\Big[R(\xi)+bt^{(p-\alpha(p-1-\beta))/(p-\alpha(p-2))}(p+\alpha(2-p))C_0^{\beta-1}(\xi_0+\xi)^{\frac{p(1-\beta)}{2-p}}\Big]\end{equation}
\begin{equation}\label{C16'''b}
R(\xi)=[\alpha-2(p-1)p^{p-1}(p+\alpha(2-p)) (2-p)^{-p} C_0^{p-2}+p(2-p)^{-1}\xi (\xi_0+\xi)^{-1}].
\end{equation}
\end{subequations}
To prove an upper estimation, we take $C_0=C_6,\;\xi_0=\xi_2$ (see Appendix). Then we have
\begin{equation}\label{C17'''}
 R(\xi)\geq \alpha (\mu_b-1)\mu_b^{-1}
\end{equation}
From \eqref{C16'''ab}, \eqref{C17'''}  it follows that 
\[\mathcal{L}_tf\geq 0 \quad \text{for}~\xi\geq 0, ~~0\leq t\leq \delta_2,\]
where 
\[\delta_2=\delta_1\quad \text{if}~b>0,\qquad \delta_2=\min(\delta_1,\delta_3)\quad \text{if}~b<0\]
and 
\[\delta_3=\big[\alpha\epsilon(A_0+\epsilon)^{1-\beta}\big(-b(p+\alpha(2-p))(1+\epsilon)\big)^{-1}\big]^{(p+\alpha(2-p))/(p+\alpha(\beta+1-p))}\]
Hence, from \eqref{C14'''ab} we have
\begin{equation}\label{C18'''}
Lg\geq 0\qquad \text{for}~~0\leq x<+\infty,~~0<t\leq \delta_2.
\end{equation}
From \eqref{C1'''} and Lemma~\ref{lemma1}, the right-hand side of \eqref{C15'} follows with $\delta=\delta_2$. To prove a lower estimation in this case, we take $C_0=C_5,\;\xi_0=\xi_1$. If $b>0$ and  $\beta<2/p$ we derive from \eqref{C16'''ab} that 
\begin{subequations}
\begin{equation*}
R(\xi)\leq\alpha+p(2-p)^{-1}-2(p-1)p^{p-1}(p+\alpha(2-p)) (p-2)^{-p} C_5^{p-2}
\end{equation*}
\begin{equation}\label{C19'''a}
=-(p+\alpha(2-p))\big((2-p)(1-\epsilon)\big)^{-1}\epsilon\end{equation}
\begin{equation}\label{C19'''b}
\mathcal{L}_tf\leq 0\qquad \text{for}~~\xi\geq 0,\;0\leq t\leq \delta_4,
\end{equation}
\end{subequations}
where $\delta_4=\min(\delta_1,\delta_5)$ and
\[\delta_5=\Big[(A_0-\epsilon)^{1-\beta}\big(b(2-p)(1-\epsilon)\big)^{-1}\epsilon\Big]^{(p+\alpha(2-p))/(p+\alpha(\beta+1-p))}.\]
From \eqref{C14'''ab} it follows that
\begin{equation}\label{C20'''}
Lg\leq 0\qquad \text{for}~~0\leq x<+\infty,\;0< t\leq \delta_4.
\end{equation}
 If either $b>0,\; \beta\geq 2/p$ or $b<0,\; \beta \geq 1$, from \eqref{C16'''ab} we have
\begin{equation*}
\mathcal{L}_tf=(p+\alpha(2-p))^{-1}C_5(\xi_1+\xi)^{\frac{2}{p-2}}\end{equation*}
\begin{subequations}
\begin{equation}\label{C21'''a}
\times\Big[R_1(\xi)+bt^{(p-\alpha(p-1-\beta))/(p-\alpha(p-2))}(p+\alpha(2-p))C_5^{\beta-1}(\xi_1+\xi)^{\frac{(2-p\beta)}{2-p}}\Big]\end{equation}\\[.2cm]
\begin{equation*}
R_1(\xi)=[\alpha-2(p-1)p^{p-1}(p+\alpha(2-p)) (2-p)^{-p} C_5^{p-2}](\xi_1+\xi)+p(2-p)^{-1}\xi 
\end{equation*}
\begin{equation}\label{C21'''b}
=-p(2-p)^{-1}\xi_1, 
\end{equation}
\end{subequations}
which again imply  \eqref{C19'''b}, where $\delta_4=\delta_1$ if $b<0,\;\delta_4=\min(\delta_1,\delta_5)$ if 
$b>0$ and 
\[\delta_5=[p\big(b(p+\alpha(2-p))(2-p)\big)^{-1}(A_0-\epsilon)^{1-\beta}]^{(p+\alpha(2-p))/(p+\alpha(\beta+1-p))}.\]
As before  \eqref{C20'''} follows from  \eqref{C21'''b}. From  \eqref{C1'''},  and Lemma~\ref{lemma1}, the left-hand side of  \eqref{C15'} follows with $\delta=\delta_4$. Thus we have proved  \eqref{C16'} with~~$\delta=\min(\delta_2,\delta_4)$.

Let $b>0,\; \beta\geq 1$. The upper estimation of \eqref{C16'} is an easy consequence of Lemma~\ref{lemma1}, since the right-hand side of it is a solution of Eq.\eqref{C1} with $b=0$. Let $b>0$ and $\beta \geq 2/p$. Now we can fix a particular value of $\epsilon=\epsilon_0$ and take $\delta=\delta(\epsilon_0)>0$ in \eqref{C15'}. Then from the left-hand side of \eqref{C15'} and \eqref{C16'},  the asymptotic result \eqref{C17'} follows. However, if $b>0,\; 1\leq \beta<2/p,$ from  \eqref{C15'} and \eqref{C16'}
it follows that for $\forall$ fixed $t\in (0,\delta(\epsilon)]$
\begin{equation*}
D(1-\epsilon)^{1/(2-p)}\leq \underset{x\rightarrow +\infty}\liminf ut^{1/(p-2)}x^{\frac{p}{2-p}}\leq \underset{x\rightarrow +\infty}\limsup ut^{1/(p-2)}x^{\frac{p}{2-p}}\leq D,
\end{equation*}
which easily implies  \eqref{C18'} in view of arbitrariness of $\epsilon$.

We now let $b<0,\; \beta\geq 1$ and prove  \eqref{C19'}. Consider a function 
\[\bar{g}(x,t)=D(1-\epsilon)^{1/(p-2)}t^{1/(2-p)}x^{p/(p-2)}\]
in $G=\{(x,t):\mu t^{1/(p+\alpha(2-p))}<x<+\infty,~~0<t\leq \delta\}$, where  $\mu$ is defined as in  \eqref{C19'}. Let $g(x,t)=\bar{g}(x,t)$ for $(x,t)\in \bar{G}\backslash (0,0)$ and $g(0,0)=0$. We have 
\[Lg=D(2-p)^{-1}(1-\epsilon)^{(1/(p-2)}t^{(p-1)/(2-p)}x^{p/(p-2)} \mathcal{G}\quad \text{in} ~G\]
\[\mathcal{G}=\epsilon+b(2-p)D^{\beta-1}(1-\epsilon)^{(\beta-1)/(p-2)}t^{(\beta+1-p)/(2-p)}x^{p(\beta-1)/(p-2)}.\]
We then derive
\[\mathcal{G}\geq\epsilon+b(2-p)D^{\beta-1}(1-\epsilon)^{(\beta-1)/(p-2)}\mu^{p(\beta-1)/(p-2)}t^{(p+\alpha(\beta+1-p))/(p+\alpha(2-p))}\quad \text{in} ~G.\]
Hence,
\begin{equation*}
\mathcal{G}\geq 0\quad \text{in} ~G,\quad\text{for}~~\delta\in(0,\delta_0]
\end{equation*}
\begin{equation*}
\delta_0=\Big[\big(-b(2-p)\big)^{-1}D^{1-\beta}(1-\epsilon)^{(1-\beta)/(p-2)}\mu^{p(1-\beta)/(p-2)}\epsilon\Big]^{(p+\alpha(2-p))/(p+\alpha(\beta+1-p))},
\end{equation*}
which implies 
\begin{subequations}\label{C22'''ab}
\begin{equation}\label{C22'''a}
Lg\geq 0\quad \text{in} ~G.\end{equation}
Moreover, we have 
\[g|_{x=\mu t^{1/(p+\alpha(2-p))}}=(A_0+\epsilon)t^{\alpha/(p+\alpha(2-p)}\quad \text{for}~~0\leq t\leq \delta.\]
From\eqref{C15'}, it follows that 
 \[ u|_{x=\mu t^{1/(p+\alpha(2-p))}}\leq C_6(\xi_2+\mu)^{\frac{p}{p-2}}t^{\alpha/(p+\alpha(2-p))} \]
 \[\leq (A_0+\epsilon)t^{\alpha/(p+\alpha(2-p)}\quad \text{for}~~ 0\leq t\leq \delta.\]
Therefore, we have
\begin{equation}\label{C22'''b} 
g\geq u\qquad \text{on}~~\bar{G}~\backslash ~G,\end{equation}
\end{subequations}
From \eqref{C22'''ab}, and Lemma~\ref{lemma1}, the desired estimation \eqref{C19'} follows. Since $\epsilon>0$ is arbitrary, from the left-hand side of \eqref{C15'} and \eqref{C19'} the asymptotic result  \eqref{C17'} follows as before.

Let $b>0,\; p-1<\beta<1$. The left-hand side of \eqref{C20'} may be proved as the left-hand side of \eqref{C3'a} was earlier. The only difference is that we take $f_1(\zeta)=C_*(1-\epsilon)(\zeta_8+\zeta)_+^{p/(p-1-\beta)}$ in \eqref{C2'''},  \eqref{C3'''ab}. The right-hand side of \eqref{C20'} is almost trivial, since $C_*x^{p/(p-1-\beta)}$ is a stationary solution of Eq. \eqref{C1}. The important point in \eqref{C20'} is that $\delta>0$ does not depend on $\epsilon>0$. This is clear from the analysis involved in the proof of the similar estimation  \eqref{C3'a}. From \eqref{C20'}, it follows that $\forall$ fixed $t\in (0,\delta]$, we have
\[C_*(1-\epsilon)\leq \underset{x\rightarrow +\infty}\liminf ux^{p/(\beta+1-p)}\leq \underset{x\rightarrow +\infty}\limsup ux^{p/(\beta+1-p)}\leq C_* .\]
Since $\epsilon>0$ is arbitrary, \eqref{C21'} easily follows.

Assume now that $b=0$. First consider the case when $u_0$ is defined by \eqref{C4}. The self-similar form \eqref{C22'} and the formula \eqref{C2'}  follow from Lemma~\ref{lemma3.1}. To prove \eqref{C23'}, consider a function $g$ from \eqref{C13'''}, which satisfies  \eqref{C14'''ab} with $b=0$. As a function $f$ we take \eqref{C15'''} with $\gamma_0=p/(2-p)$. Then we drive  \eqref{C16'''ab} with $b=0$. To prove an upper estimation we take $C_0=C_7,\; \xi_0=\xi_4$ and from \eqref{C16'''b} we have 
\begin{equation*}
R(\xi)\geq\big[\alpha-2p^{p-1}(p-1)(p+\alpha(2-p))(2-p)^{-p}C_7^{p-2}\big]=0,
\end{equation*}
which implies  \eqref{C18'''} with $\delta_2=+\infty.$ As before, from \eqref{C18'''} and Lemma~\ref{lemma1}, the right-hand side of \eqref{C23'}  follows. The left-hand side of \eqref{C23'} may be established similarly if we take $ C_0=D,\xi_0=\xi_3$.
To prove the estimation \eqref{C16'}, consider 
\[g_{\mu}(x,t)=D(t+\mu)^{1/(2-p)}(x+\mu)^{p/(p-2)},\qquad \mu>0,\]
which is a solution of Eq.\eqref{C1} for $x>0,~t>0$. Since 
\[g_{\mu}(0,t)\geq D\mu^{(p-1)/(p-2)}\geq u(0,t)\quad \text{for}~~ 0\leq t\leq T(\mu)=[DA_0^{-1}\mu^{(p-1)/(p-2)}]^{(p+\alpha(2-p))/\alpha},\]
the comparison Lemma~\ref{lemma1} implies 
\[u(x,t)\leq g_{\mu}(x,t)\qquad \text{for}~~ 0<x<+\infty,~0\leq t\leq T(\mu).\]
In the limit as $\mu\rightarrow 0+$, we can easily derive \eqref{C16'}. 
Finally, from \eqref{C23'} and \eqref{C16'} it easily follows that for an arbitrary fixed $0<t<+\infty$, the asymptotic formula \eqref{C17'} is valid.  If $u_0$ satisfies \eqref{C3} with $\alpha>0$, then \eqref{C1'} and \eqref{C1'''} follow from Lemma \ref{lemma3.1}. Similarly, we can then prove that for an arbitrary sufficiently small $\epsilon >0$ there exists a $\delta=\delta(\epsilon)>0$ such that \eqref{C23'} is valid for $0\leq t \leq \delta(\epsilon)$, except that in the left-hand side  (respectively in the right-hand side) of \eqref{C23'} the constant $A_0$ is replaced by $A_0-\epsilon$ (respectively by  $A_0+\epsilon$ ). Then we can fix a particular value of $\epsilon=\epsilon_0$ and let $\delta=\delta(\epsilon_0)>0$. Obviously, from the local analog of \eqref{C23'}  and \eqref{C16'} it follows that, for arbitrary fixed $t\in (0,\delta]$, the asymptotic formula \eqref{C17'} is valid. \end{proof}

\section{Conclusion}
This paper presents a full classification of the short-time behavior of the interfaces and local solutions near the interfaces or at infinity in the Cauchy problem for the 
nonlinear parabolic $p$-Laplacian type reaction-diffusion equation of non-Newtonian elastic filtration in the fast diffusion regime:
\[ u_t=\Big(|u_x|^{p-2}u_x\Big)_x-bu^{\beta}=0, \ x\in \mathbb{R}, 0<t<T, \ 1<p<2, \beta >0; \ u(x,0)\sim C(-x)_+^\alpha,~~\text{as} \ x\to 0^-, \]
and either $b\geq 0$ or $b<0,\beta\geq 1$.
The classification is based on the relative strength of the diffusion and absorption forces. The following are the main results:
\begin{itemize}
\item If $b>0, 0<\beta<p-1,\; 0<\alpha<p/(p-1-\beta)$, then diffusion weakly dominates over the absorption and the interface expands with asymptotics
\[ \eta(t)\sim \gamma(C,p,\alpha)t^{(p-1-\beta)/p(1-\beta)} \quad\text{as} \ t\to 0^+. \]
\item If ~$b>0, 0<\beta <p-1,\ \alpha=p/(p-1-\beta)$, then diffusion and absorption are in balance, and there is a critical value $C_*$ such that the interface expands or shrinks accordingly as $C>C_*$ or $C<C_*$ and
\[ \eta(t) \sim\zeta_*(C,p) t^{(p-1-\beta)/p(1-\beta)},\qquad \text{as} \ t\rightarrow 0^+,\]
where $\zeta_* \lessgtr 0$ if $C\lessgtr C_*$.
\item If $b>0, 0<\beta<p-1, \alpha >p/(p-1-\beta)$, then absorption strongly dominates over diffusion and the interface shrinks with asymptotics
\[
\eta(t)\sim -\ell_*(C,\alpha,p,\beta) t^{1/\alpha(1-\beta)}\qquad\text{as}~t\rightarrow 0^+,
\]
\item $b>0, 0<\beta=p-1<1, \alpha>0$, then domination of the diffusion over absorption is moderate, there is an infinite speed of propagation, and the solution has exponential decay at infinity.
\item If either $b>0,\beta>p-1$ or $b<0,\beta\geq 1$, then diffusion strongly dominates over the absorption, and the solution has power type decay at infinity independent of $\alpha >0$,
which coincides with the asymptotics of the fast diffusion equation ($b=0$).
\end{itemize}

\section*{Acknowledgment}
The authors thank anonymous referee for valuable comments and suggestions which improved the paper.
%Insert the Acknowledgment text here.


%\bibliographystyle{imamat}
%\bibliography{sample}

\begin{thebibliography}{}

\bibitem{Abdulla3}Abdulla, U.G. (2000) Reaction--diffusion in irregular domains. {\em Journal of Differential Equations},
\textbf{164}, 321--354.
\bibitem{Abdulla1}
Abdulla, U.G. {\&} King, J.œR. (2000) Interface development and local solutions to reaction-diffusion equations.
{\em SIAM Journal on Mathematical Analysis}, \textbf{32}, 235--260.
\bibitem{Abdulla2}Abdulla, U.œG. (2002) Evolution of interfaces and explicit asymptotics at infinity for the fast diffusion equation with absorption.
{\em Nonlinear Analysis: Theory, Methods \& Applications}, \textbf{50}, 541--560.
\bibitem{Abdulla5}
Abdulla, U.œG. (2001) On the Dirichlet problem for the nonlinear diffusion equation in non-smooth domains.
{\em Journal of Mathematical Analysis and Applications}, \textbf{260}, 384--403.
\bibitem{Abdulla6}
Abdulla, U.G. (2005) Well-posedness of the {D}irichlet problem for the non-linear
              diffusion equation in non-smooth domains.
{\em Trans. Amer. Math. Soc.}, \textbf{357}, 247--265.
%\bibitem{Abdulla7}
%Abdulla, U.G. (2006) Reaction-Diffusion in Nonsmooth and Closed Domains.
%{\em Boundary Value Problems}, \textbf{2007}, 031261.
\bibitem{Abdulla35}
Abdulla, U.G. {\&} Jeli, R. (2017) Evolution of interfaces for the non-linear parabolic
              {$p$}-{L}aplacian type reaction-diffusion equations.
{\em European Journal of Applied Mathematics}, \textbf{28}, 5, 827--853.
\bibitem{ADS}
Antontsev, S.~N. {\&} D{\'\i}az, J.~I. {\&} Shmarev, S. (2012) Energy methods for free boundary problems: Applications to nonlinear PDEs and fluid mechanics.
 \textbf{48}. Springer Science \& Business Media.
 \bibitem{Barenblatt1}
Barenblatt, G.~I. (1952) On some unsteady motions of a liquid and gas in a porous medium.
{\em Prikl. Mat. Mekh}, \textbf{16}, 67--78.
\bibitem{Barenblatt2}
Barenblatt, G.~I. (1996) Scaling, self-similarity, and intermediate asymptotics.
{\em Cambridge Texts in Applied Mathematics}, \textbf{14}, xxii+386. Cambridge University Press, Cambridge.
%\bibitem{degtyarev2012solvability}
%Degtyarev, S.~P. {\&} Tedeev, A.~F. (2012) On the solvability of the Cauchy problem with growing initial data for a class of anisotropic parabolic equations.
%{\em Journal of Mathematical Sciences}, \textbf{181}, 28--46.
\bibitem{diaz} Diaz, J.I. (1985) Nonlinear partial differential equations and free boundaries. Volume I. {\em Research Notes in Mathematics, {\bf106}, Pitman, Boston, 323 pp.}
\bibitem{dibe1}
DiBenedetto, E. (1986) On the local behaviour of solutions of degenerate parabolic
              equations with measurable coefficients.
{\em Ann. Scuola Norm. Sup. Pisa Cl. Sci. (4)}, \textbf{13}, 487--535.
\bibitem{dibe2}
DiBenedetto, E. {\&} Herrero, M.~A. (1989) On the Cauchy problem and initial traces for a degenerate parabolic equation.
{\em Transactions of the American Mathematical Society}, \textbf{314}, 187--224.
\bibitem{dibe3}
DiBenedetto, E. {\&} Herrero, M.~A. (1990) Nonnegative solutions of the evolution {$p$}-{L}aplacian
              equation. {I}nitial traces and {C}auchy problem when
              {$1<p<2$}.
{\em Arch. Rational Mech. Anal.}, \textbf{111}, 225--290.
\bibitem{dibe-sv}
DiBenedetto, E. (1993) Degenerate parabolic equations.
{\em Universitext}, xvi+387. Springer-Verlag, New York.

\bibitem{est-vazquez}
Esteban, J.~R. {\&} V{\'a}zquez, J.~L. (1986) On the equation of turbulent filtration in one-dimensional porous media.
{\em Nonlinear Analysis: Theory, Methods \& Applications}, \textbf{10}, 1303--1325.
\bibitem{Grundy1}
Grundy, R.~E. {\&} Peletier, L.~A. (1987) Short time behaviour of a singular solution to the heat equation with absorption.
{\em Proc. Roy. Soc. Edinburgh Sect. A}, \textbf{107}, 271--288.
\bibitem{Grundy2}
Grundy, R.~E. {\&} Peletier, L.~A. (1990) The initial interface development for a reaction-diffusion equation with power-law initial data.
{\em Quarterly journal of mechanics and applied mathematics}, \textbf{43}, 535--559.
%\bibitem{herrero1982propagation}
%Herrero, M.œA. {\&} Vazquez, J.œL. (1982) On the propagation properties of a nonlinear degenerate parabolic equation.
%{\em Communications in Partial Differential Equations}, \textbf{7}, 1381--1402.
%\bibitem{ishige1996existence}
%Ishige, K. (1996) On the existence of solutions of the Cauchy problem for a doubly nonlinear parabolic equation.
%{\em SIAM Journal on Mathematical Analysis}, \textbf{27}, 1235--1260.
\bibitem{kalashnikov1987some}
Kalashnikov, A.~S. (1987) Some problems of the qualitative theory of non-linear degenerate second-order parabolic equations.
{\em Russian Mathematical Surveys}, \textbf{42}, 169.
\bibitem{kalashnikov2}
Kalashnikov, A.~S. (1978) On a nonlinear equation appearing in the theory of non-stationary filtration.
{\em Trudy. Sem. Petrovsk}, \textbf{5}, 60--68.
\bibitem{kalashnikov3}
Kalashnikov, A.~S. (1982) Propagation of perturbations in the first boundary value problem for a degenerate parabolic equation with a double nonlinearity.
{\em Trudy Sem. Petrovsk.}, \textbf{8}.
%\bibitem{king1995development}
%King, J.~R. (1995) Development of singularities in some moving boundary problems.
%{\em European Journal of Applied Mathematics}, \textbf{6}, 491--507.
%\bibitem{Lidumu}
%Li Z. {\&} Du, W. {\&} Mu, C. (2013) Travelling-wave solutions and interfaces for non-Newtonian diffusion equations with strong absorption.
%{\em Journal of Mathematical Research with Applications}, \textbf{33}, 451--462.
\bibitem{shmarev2015interfaces}
Shmarev, S. {\&} Vdovin, V. {\&} Vlasov, A. (2015) Interfaces in diffusion--absorption processes in nonhomogeneous media.
{\em Mathematics and Computers in Simulation}, \textbf{118}, 360--378.
\bibitem{tsutsumi}
Tsutsumi, M. (1988) On solutions of some doubly nonlinear degenerate parabolic equations with absorption. {\em Journal of mathematical analysis and applications}, \textbf{132}, 187--212.
%\bibitem[V\'azquez, 1984]{V.J.}
%V\'azquez, J.~L. (1984) Behaviour of the velocity of one-dimensional flows in porous
%              media. {\em Trans. Amer. Math. Soc.}, \textbf{286}, 787--802.
%\bibitem{HerreroPierre}
%Herrero, M.~A. {\&} Pierre, M. (1985) The {C}auchy problem for $u_t=\Delta u^m$ when $0<m<1$. {\em Transactions of the American Mathematical Society}, \textbf{291}, 145--158.

\end{thebibliography}
%
% once the .bbl file has been generated then place the text in your article.

% To get the numbered reference style the author should use [numbib]
%as an option in the document class.  For example: \documentclass[numbib]{imamat}

\clearpage

\appendix

\section*{Appendix}
\label{app1}
We give here explicit values of the constants used in Section \ref{sec: description of the main results} in the outline of the results and later in Section \ref{sec: proofs of the main results.} during the proof of these results. \\
(1)   $0<\beta<p-1,\; 0<\alpha<p/(p-1-\beta)$

\[C_1=\big((1-\beta)/(2-p)\big)^{1/(p-1-\beta)}C_*\]
\[\zeta_1= b^{\frac{p-2}{p(1-\beta)}}\big(p^{p-1}(p-1)\big)^{1/p}(1+\beta)^{1/p}\big(p-1-\beta\big)^{\frac{\beta(p-1)-1}{p(1-\beta)}}\big((2-p)/(1-\beta)\big)^{(2-p)/p(1-\beta)},\]
 \[\zeta_2=b^{(p-2)/p(1-\beta)}(p-1)^{1/p}p^{(p-1)/p}(1+\beta)^{(2-p)/p(1-\beta)}2^{(p-1-\beta)/p(1-\beta)}(2-p)^{\frac{\beta(p-1)-1}{p(1-\beta}}(1-\beta)(p-1-\beta)^{-1}\]
         \[ \ell_0=\frac{p-1-\beta}{1-\beta}\zeta_2,\]
 (2)~$b>0,\;0<\beta<1,\; \beta <p-1<\beta^{-1},\; \alpha=p(p-1-\beta)^{-1}$\\
 
$\zeta_3=A_1^{\frac{p-2}{p}}\big((1-\beta)(1+\beta)p^{p-1}(p-1)\big)^{\frac{1}{p}}\big(1+b(1-\beta)A_1^{\beta-1}\big)^{-\frac{1}{p}}(p-1-\beta)^{-1},~~~$

$\zeta_4=\Big(A_1/C_*\Big)^{\frac{p-1-\beta}{p}},\qquad C_2=A_1 \zeta_3^{-\frac{p}{p-1-\beta}}$,

$\zeta_5=\ell_1-(\lambda/C_*)^{(p-1-\beta)/p}>0 ~~(\text{see Lemma \ref{lemma3.3} and } \eqref{C7''})$

$\ell_2=C^{\frac{1+\beta-p}{p}}\Big[b(1-\beta)(\delta_* \Gamma)^{-1}\Big((1-\delta_* \Gamma)-\big(1-\delta_* \Gamma\big)^{1-p}\Big(C/C_*\Big)^{p-1-\beta}\Big) \Big]^{\frac{p-1-\beta}{p(1-\beta)}}$,

$\zeta_6=\delta_* \Gamma \ell_2,~~~~~~~~\Gamma=1- (C/C_*)^{\frac{p-1-\beta}{p}},~~~~~~~~C_3=C \big(1-\delta_* \Gamma\big)^{\frac{p}{1+\beta-p}},$~
where ~$\delta_*\in(0,1)$~satisfies

$g(\delta_*)=\underset{[0;1]}\max g(\delta),\quad ~~~g(\delta)=\delta^{\frac{1+\beta(1-p)}{p(1-\beta)}}\Big[(1-\delta \Gamma)-\Big(C/C_*\Big)^{p-1-\beta}\big(1-\delta \Gamma\big)^{1-p}\Big) \Big],$\\\\
(5) $\beta>p-1$\\

$D=\Big[\frac{2(p-1)p^{p-1}}{(2-p)^{p-1}}\Big]^{1/(2-p)}$\\

$\xi_1=(A_0-\epsilon)^{(p-2)/p}(1-\epsilon)^{1/p}D^{(2-p)/p}$    if  $b>0,\; 1\leq \beta <2/p$,\\

$\xi_1= (A_0-\epsilon)^{(p-2)/p}D^{(2-p)/p}$ if either $b>0,\; \beta\geq 2/p$ or $b<0,\; \beta \geq 1$,\\

$C_5=(A_0-\epsilon)\xi_1^{p/(2-p)}$\\

$A_0=f(0)>0\qquad$(see \eqref{C2'} and lemma \ref{lemma3.1})\\

$\xi_2=(A_0+\epsilon)^{(p-2)/p}\big[\frac{2(p-1)p^{p-1}(p+\alpha(2-p))\mu_b}{ \alpha(2-p)^{p}}\big]^{1/p}$\\

$C_6=\big[\frac{2(p-1)p^{p-1}(p+\alpha(2-p))\mu_b}{ \alpha(2-p)^{p}}\big]^{1/(2-p)}$\\

$\mu_b=1$   if  $b>0$, \qquad $\mu_b=1+\epsilon$   if   $b<0$,\\

$\zeta_8=\Big[b(1-\beta)C_*^{\beta-1}(1-\epsilon)^{\beta-1}\big((1-\epsilon)^{p-1-\beta}-1\big)\Big]^{(p-1-\beta)/p(1-\beta)}$\\\\

$\xi_3=(A_0/D)^{(p-2)/p},\qquad \qquad\xi_4=\xi_3\big(1+p/\alpha(2-p)\big)^{1/p}$\\

$C_7=D\big[1+p/\alpha(2-p)\big]^{1/(2-p)}$\\
\end{document}